\documentclass[11pt]{amsart}
\usepackage[margin=1.2in]{geometry}
\usepackage{graphicx}
\usepackage{amssymb}

\usepackage{hyperref}
\usepackage{xcolor}

\usepackage{bm}

\numberwithin{equation}{section}

\newtheorem{thm}{Theorem} [section]
\newtheorem{lemm}{Lemma} [section]
\newtheorem{coro}{Corollary} [section]
\newtheorem{prop}{Proposition} [section]
\newtheorem{remark}{Remark}[section]

\newtheorem{claim}{Claim}[section]

\newcommand\bN{\mathbb N}
\newcommand\bR{\mathbb R}
\newcommand\bT{\mathbb T}

\title[ compressible MHD boundary layer]{Local-in-time well-posedness  for Compressible MHD boundary layer}
\author[Yongting Huang]{Yongting Huang}
\address{Yongting Huang
	\newline\indent
	Department of Mathematics,
	City University of Hong Kong,
	Tat Chee Avenue, Kowloon, Hong Kong}
\email{ythuang7-c@my.cityu.edu.hk}
\author[Cheng-Jie Liu]{Cheng-Jie Liu}
\address{Cheng-Jie Liu
\newline\indent
Institute of Natural Sciences, Shanghai Jiao Tong University, Shanghai, China}
\email{liuchengjie@sjtu.edu.cn}

\author[Tong Yang]{Tong Yang}
\address{Tong Yang
\newline\indent School of Mathematical Sciences, Shanghai Jiao Tong University, Shanghai, 200240, P.R. China
\newline\indent
Department of Mathematics,
City University of Hong Kong,
Tat Chee Avenue, Kowloon, Hong Kong
}
\email{matyang@cityu.edu.hk}


\date{}


\begin{document}
\bibliographystyle{elsart-num-sort}
\maketitle

\begin{abstract}
In this paper, we are concerned with the motion of electrically conducting fluid governed by
the two-dimensional non-isentropic viscous compressible MHD system on the half plane,  
with no-slip condition for velocity field, perfect conducting condition for magnetic field and Dirichlet boundary condition for temperature on the boundary. When the viscosity,
heat conductivity and magnetic diffusivity coefficients tend to zero in the same rate, there is a boundary layer that is described by a Prandtl-type system. By applying a
coordinate transformation in terms of stream function as motivated by
the recent work \cite{liu2016mhdboundarylayer} on the incompressible MHD system, under the non-degeneracy condition on the tangential magnetic field, we obtain the local-in-time well-posedness
of the boundary layer system in weighted Sobolev spaces. 


\end{abstract}

\section{Introduction and main result}

The compressible magnetohydrodynamics (MHD) system describes the mutual interaction of fluid and magnetic field where the velocity
and magnetic field are governed by the compressible Navier-Stokes equations
coupled with the Maxwell equations. In two-dimensional (2D) case, it takes the form as follows:
\begin{align}\label{mhd}
\left\{
\begin{aligned}
&\displaystyle \partial_t\rho +\nabla\cdot(\rho \mathbf{u})=0,\\
&\displaystyle \rho\big(\partial_t \mathbf{u}+(\mathbf{u}\cdot\nabla)\mathbf{u}\big)
+\nabla p(\rho,\theta)-(\nabla\times \mathbf{H})\times \mathbf{H}
=\mu\Delta\mathbf{u}+(\lambda+\mu)\nabla(\nabla\cdot\mathbf{u}),\\
&\displaystyle c_V\rho\big(\partial_t\theta+(\mathbf{u}\cdot\nabla)\theta\big)+p(\rho,\theta)\nabla\cdot\mathbf{u}
=\kappa\Delta\theta+\lambda(\nabla\cdot\mathbf{u})^2+2\mu |\mathfrak{D}(\mathbf{u})|^2
+\nu|\nabla\times \mathbf{H}|^2,\\
&\displaystyle \partial_t \mathbf{H}-\nabla\times (u\times \mathbf{H})=\nu\Delta \mathbf{H},\qquad \nabla\cdot \mathbf{H}=0.
\end{aligned}
\right.
\end{align}
Here the unknown functions $\rho,\ \mathbf{u}=(u_1,u_2)^{\it T},\ \theta$ and $\mathbf{H}=(h_1,h_2)^{\it T}$ represent the density, velocity, absolute temperature and magnetic field, respectively. $p(\rho,\theta)$ denotes the pressure of fluid. $c_V$ is the specific heat capacity; $\lambda$ and $\mu$ are the viscosity coefficients satisfying $\mu>0,\lambda+\mu>0$; $\kappa$ is the heat conductivity coefficient and $\nu$ is the magnetic diffusivity coefficient.
The deformation tensor $\mathfrak{D}(\mathbf{u})$ is given by
\[
\displaystyle \mathfrak{D}(\mathbf{u})=\frac{1}{2}(\nabla \mathbf{u}+(\nabla \mathbf{u})^{\it T}).
\]
We are concerned with the initial boundary value problem for (\ref{mhd}) in a periodic domain $\{(t,x,y)|t\in[0,T],x\in\bT,y\in\bR_+\},$
with no-slip boundary condition on velocity, Dirichlet boundary condition for temperature and perfect conducting boundary condition for magnetic field, that is,
\begin{equation}\label{bou-con}
\mathbf{u}|_{y=0}=\mathbf{0},\quad \theta|_{y=0}=\theta^\ast(t,x),\quad (\partial_yh_1,h_2)|_{y=0}=\mathbf{0}.
\end{equation}
For simplicity, we consider the
ideal gas, i.e.,
\begin{align}\label{def-pressure}
p=p(\rho,\theta)=R\rho\theta,
\end{align}
with some constant $R>0$.

In this paper, we are interested in the asymptotic behavior of solutions to the problem (\ref{mhd})-(\ref{bou-con}) as the viscosities $\mu, \lambda$, heat conductivity $\kappa$ and magnetic diffusivity $\nu$ tend to zero.
Formally, when all these physical parameters vanish, the  system (\ref{mhd}) is
reduced to the ideal compressible MHD system:
\begin{equation}\label{euler}
\left\{
\begin{aligned}
&\partial_t\rho^e +\nabla\cdot(\rho^e \mathbf{u}^e)=0,\\
&\rho^e\big(\partial_t \mathbf{u}^e+(\mathbf{u}^e\cdot\nabla)\mathbf{u}^e\big)+\nabla p(\rho^e,\theta^e)
-(\nabla\times \mathbf{H}^e)\times \mathbf{H}^e=0,\\
&c_V\rho^e\big(\partial_t\theta^e+(\mathbf{u}^e\cdot\nabla)\theta^e\big)
+p(\rho^e,\theta^e)\nabla\cdot\mathbf{u}^e=0,\\
&\partial_t \mathbf{H}^e-\nabla\times (u^e\times \mathbf{H}^e)=0,\qquad \nabla\cdot \mathbf{H}^e=0.
\end{aligned}
\right.
\end{equation}
Corresponding to the no-slip boundary condition on velocity field and the perfectly conducting boundary condition on magnetic field in (\ref{bou-con}), a natural boundary condition for system (\ref{euler}) is
\begin{equation}\label{impermeable}
(u_2^e,h_2^e)|_{y=0}=\mathbf{0}.
\end{equation}
In fact, under this impermeable boundary condition (\ref{impermeable}), the boundary $\{y=0\}$ is a particle-path so that the tangential velocity $u_1^e$, the temperature $\theta^e$ and the tangential magnetic $h_1^e$ are determined by the initial data at $\{t=0\}$.

The inconsistency of boundary conditions between (\ref{impermeable}) and (\ref{bou-con}) gives rise to a thin transition layer near the physical boundary $\{y=0\}$ in the vanishing viscosity, heat conductivity and magnetic diffusivity limit process, cf.~\cite{prandtl1904boundarylayer}.
Precisely, this transition layer is called the boundary layer, in which the behavior of the tangential velocity, temperature and tangential magnetic field change dramatically from $(u_1,\theta,\partial_yh_1)=(0,\theta^\ast(t,x),0)$ to $(u_1^e,\theta^e,h_1^e)(t,x,0)$ determined by (\ref{euler})-(\ref{impermeable}).

Before exhibiting the description of behavior of boundary layers, let us state some known results first.
The boundary layer theory, introduced by L. Prandtl, cf.~\cite{prandtl1904boundarylayer} for the incompressible Navier-Stokes equations with no-slip boundary condition, is concerned with the asymptotic analysis of the behavior of fluid inside the boundary layer, which can be described by the so-called Prandtl equations.
From a mathematical point of view, the first systematic study of Prandtl equations was developed by Oleinik~\cite{oleinik1963prandtl} with
detailed explanation in~\cite{oleinik1999boundary}, and under the monotonicity condition on tangential velocity with respect to the normal variable to the boundary, the local (in time) well-posedness of Prandtl equations was obtained by using the Crocco transformation. Furthermore, with an additional favorable condition on the pressure, a global (in time) weak solution was obtained by Xin-Zhang~\cite{xin2004prandtl}.
Recently, Alexander-Wang-Xu-Yang~\cite{alexandre2015prandtl} and Masmoudi-Wong~\cite{masmoudi2015prandtl} proved the well-posedness result independently in the framework of weighted Sobolev space by using energy method. Accordingly,
when the monotonicity assumption on the velocity is violated,  
the Prandtl equations are ill-posed in finite order Sobolev spaces, cf.~\cite{eweinan2000navier,eweinan1997prandtl,gargano2009prandtl,gerard2010prandtl, gerard2012remarks, grenier2000instability, grenier2016spectral,grenier2016shearflow,hong2003prandtl}. Note that the above results are basically limited to the case of two space variables, and there are few results with respect to three space variables, cf.~\cite{liu2016prandtl,  liu2016weakprandtl,liu2017prandtl}. It's worth noting that without the monotonicity assumption, the well-posedness theory of Prandtl equations can also be established in the framework of analytic functions, cf.~\cite{sammartino1998navier,caflisch2000prandtl,lombardo2003boundary,ignatova2016prandtl,zhang2016prandtl} or in the Gevrey class, cf.~\cite{gerard2013prandtl,gerard2016gevrey,li2016gevrey,li2016prandtl}, with respect to both two and three space variables.

For compressible fluids, the boundary layers are complicated, since in general, heat transfers rapidly near the boundary and it gives rise to the thermal layer. Moreover, there is interaction between viscous layer and thermal layer, which leads to a more complicated structure of boundary layers,  cf.~\cite{karman1938boundary,cope1948laminar,laurmann1951boundary}.
Limited to isentropic viscous fluids, where there is no thermal layer, the boundary layer can be described by a Prandtl-type system, and its
 well-posedness theory with respect to two space variables is established by Wang-Xie-Yang~\cite{wang2015navier} and Gong-Guo-Wang~\cite{gong2016boundary} independently under the same monotonicity condition as the Prandtl equations.
Very recently, Liu-Wang-Yang~\cite{liu2017thermallayer} studied the behavior of both viscous layer and thermal layer for two-dimensional non-isentropic compressible fluids in different viscosity and heat conductivity limits.

When the fluid is coupled with magnetic field, the boundary layer phenomenon is different since boundary layers for magnetic field may exist, cf. \cite{gerard2017mhd}. In particular in the case that both the Reynolds number and the magnetic Reynolds number tend to infinity at the same rate, i.e., finite magnetic Prandtl number,  the boundary layer equations derived from the MHD system are quite different from the classical Prandtl equations. A notable example is that, for two-dimensional incompressible MHD system with no-slip boundary condition on velocity and perfectly conducting boundary condition on magnetic field, as long as the tangential component of magnetic field is nondegeneracy, a well-posedness theory for boundary layer problems is established by Liu-Xie-Yang~\cite{liu2016mhdboundarylayer, liu2016mhdboundarylayer2} without monotonicity assumption on tangential velocity (i.e., recirculation is allowed within the boundary layer). This reveals
mathematically the fact that magnetic field has a stabilizing effect on boundary layers, which could provide a mechanism for containment of, for instance, the high temperature gas.

In this paper, we are concerned with the case that viscosities, heat conductivity and resistivity parameters are of the same rate. To this end, by using a small parameter $0<\varepsilon\ll1$, we set
\[
(\mu,\lambda,\kappa,\nu)=\varepsilon(\tilde\mu,\tilde\lambda,\tilde\kappa,\tilde\nu),
\]
for some constants $\tilde\mu,\ \tilde\lambda,\ \tilde\kappa,\ \tilde\nu>0$. Therefore, based on the Prandtl's assertions, the thickness of boundary layer shall be of order $\sqrt{\varepsilon}$, and its behavior is described by the Prandtl-type equations, which can be derived from the compressible MHD system (\ref{mhd})-(\ref{bou-con}). More precisely,
inside the boundary layer, set
\[
\displaystyle \tilde{t}=t,\quad \tilde{x}=x,\quad \tilde{y}=\frac{y}{\sqrt{\varepsilon}},
\]
and the new unknown functions 
\[\displaystyle 
(\tilde{\rho},\tilde{u}_1,\tilde{\theta},\tilde{h}_1)(\tilde{t},\tilde{x},\tilde{y})=(\rho,u_1,\theta,h_1)(t,x,y),
\quad(\tilde{u}_2,\tilde{h}_2)(\tilde{t},\tilde{x},\tilde{y})=\frac{1}{\sqrt{\varepsilon}}(u_2,h_2)(t,x,y).
\]
Plug this ansatz into the equations (\ref{mhd}), the leading order term gives
\begin{equation}\label{lead}
\left\{
\begin{aligned}
&\partial_t\rho+(u_1\partial_x+u_2\partial_y)\rho+\rho(\partial_x u_1+\partial_y u_2)=0,\\
&\rho\{\partial_tu_1+(u_1\partial_x+u_2\partial_y)u_1\}+\partial_x(p+\frac{1}{2}h^2_1)-(h_1\partial_x+h_2\partial_\eta) h_1=\mu\partial_{y}^2u_1,\\
&\partial_y(p+\frac{1}{2}h^2_1)=0,\\
&c_V\rho\{\partial_t\theta+(u_1\partial_x+u_2\partial_y)\theta\}+p(\partial_x u_1+\partial_y u_2)
=\kappa\partial_{y}^2\theta+\mu(\partial_y u_1)^2+\nu(\partial_y h_1)^2,\\
&\partial_th_1+\partial_y(u_2h_1-u_1h_2)=\nu\partial_{y}^2h_1,\\
&\partial_th_2-\partial_x(u_2h_1-u_1h_2)=\nu\partial_{y}^2h_2,\\
&\partial_xh_1+\partial_y h_2=0,
\end{aligned}
\right.
\end{equation}
in the region $\displaystyle \{(t,x,y)|t\in[0,T],x\in\bT,y\in\bR_+\}$, where the tildes are omitted for simplicity of notations.
 Meanwhile,  the boundary conditions (\ref{bou-con}) become
\begin{equation}\label{lead-bou}
(u_1,u_2,\partial_y h_1,h_2)|_{y=0}=\mathbf{0},\quad \theta|_{y=0}=\theta^\ast(t,x).
\end{equation}
In addition, the boundary layer  should match the corresponding  outflow, which is actually the trace of ideal MHD flow satisfying the system \eqref{euler}-\eqref{impermeable}. In other words, one has the following far-field conditions:
\begin{align}\label{farfield}
	\lim\limits_{y\rightarrow+\infty}(\rho, u_1, \theta, h_1)(t,x,y)=(\rho^e, u_1^e, \theta^e, h_1^e)(t,x,0).
\end{align}
Thus, it turns out the boundary layer satisfies the boundary value problem \eqref{lead}-\eqref{farfield}. 

Let us simplify the equations \eqref{lead} first. The equation $(\ref{lead})_3$ says that the leading order of the total pressure $\displaystyle p+\frac{1}{2}h^2_1$ is invariant cross the boundary layer, and then it should match the outflow pressure, 
i.e., the trace of the total pressure of the ideal MHD equations (\ref{euler}). Consequently, it yields by \eqref{def-pressure} and \eqref{farfield} that
\begin{equation}\label{pressure}
(p+\frac{1}{2}h^2_1)(t,x,y)\equiv\big(R\rho^e\theta^e+\frac{1}{2}(h^e_1)^2\big)(t,x,0)=:P(t,x).
\end{equation}
Then, since we are concerned with the case $\rho, \theta>0$, \eqref{def-pressure} and \eqref{pressure} give
\begin{align}\label{def-rho}
	\rho(t,x,y)=\frac{2P(t,x)-h_1^2(t,x,y)}{2R\theta(t,x,y)},
\end{align}
which implies
\begin{align}\label{positive-p}
	2P(t,x)-h_1^2(t,x,y)>0.
\end{align}

Now, thanks to \eqref{def-rho} one can eliminate the unknown function $\rho(t,x,y)$ from the equations \eqref{lead}. Precisely, $(\ref{lead})_5$ can be rewritten by $(\ref{lead})_7$ as
\begin{equation}\label{leadh1}
\partial_th_1+(u_1\partial_x+u_2\partial_y)h_1+h_1(\partial_x u_1+\partial_y u_2)
=\nu\partial_{y}^2h_1+(h_1\partial_x+h_2\partial_y)u_1.
\end{equation}
Then, substituting \eqref{def-rho} into $\eqref{lead}_1$,  
it follows that by using$\eqref{lead}_4$ and \eqref{leadh1},
\begin{align}\label{div}
\partial_x u_1+\partial_y u_2=~&\frac{(1-a)}{Q}h_1\big[(h_1\partial_x+h_2\partial_y)u_1+\nu\partial_y^2h_1\big]-\frac{(1-a)}{Q}( P_t+ P_x u_1)\nonumber\\
&
+\frac{a}{Q}\big[\kappa\partial_{y}^2\theta+\mu(\partial_y u_1)^2+\nu(\partial_y h_1)^2\big],
\end{align}
where the positive constant $\displaystyle a:=\frac{R}{c_V+R}<1$ and
\begin{align*}
Q=Q(t,x,y):=	P(t,x)+\frac{1}{2}(1-2a)h_1^2(t,x,y),
\end{align*}
provided the fact from \eqref{positive-p}:
\begin{align*}
	Q(t,x,y)>0.
\end{align*}

Next, plugging \eqref{pressure} and \eqref{def-rho} into $\eqref{lead}_2$, \eqref{def-rho} and \eqref{div} into $\eqref{lead}_4$,  \eqref{div} into \eqref{leadh1} respectively, one can rewrite the equations for $u_1, \theta$ and $h_1$. Also, it is noted that $(\ref{lead})_6$ is a direct consequence of $(\ref{lead})_5, (\ref{lead})_7$ and the boundary condition
$h_2|_{y=0}=0$ in (\ref{lead-bou}). Therefore, the boundary value problem \eqref{lead}-\eqref{farfield} is reduced to the equation \eqref{def-rho} for $\rho$ coupled with the following boundary value problem for $(u_1, u_2, \theta, h_1, h_2)$ in $\{(t,x,y)|t\in[0,T],x\in\bT,y\in\bR_+\}$,
\begin{equation}\label{simp-lead-noP}
\left\{
\begin{aligned}
&\partial_tu_1+(u_1\partial_x+u_2\partial_y)u_1-\frac{R\theta}{P-\frac{1}{2}h^2_1}(h_1\partial_x+h_2\partial_y)h_1+\frac{RP_x \theta}{P-\frac{1}{2}h^2_1}
=\frac{\mu R\theta}{P-\frac{1}{2}h^2_1}\partial_{y}^2u_1,\\
&\partial_t\theta+(u_1\partial_x+u_2\partial_y)\theta
+\frac{a\theta h_1}{Q}(h_1\partial_x+h_2\partial_y)u_1-\frac{a(P_t +P_x u_1)\theta}{Q}\\
&\qquad=\frac{a\theta(P+\frac{1}{2}h_1^2)}{Q(P-\frac{1}{2}h^2_1)}
\big[\kappa\partial_{y}^2\theta+\mu(\partial_y u_1)^2+\nu(\partial_y h_1)^2\big]-\frac{a\nu\theta h_1}{Q}\partial_{y}^2h_1,\\
&\partial_t h_1+(u_1\partial_x+u_2\partial_y)h_1-\frac{P-\frac{1}{2}h_1^2}{Q}(h_1\partial_x+h_2\partial_y)u_1-\frac{(1-a)(P_t+P_x u_1)h_1}{Q}\\
&\qquad=\frac{\nu(P-\frac{1}{2}h_1^2)}{Q}\partial_{y}^2h_1-\frac{ah_1}{Q}\big[\kappa\partial_{y}^2\theta+\mu(\partial_y u_1)^2+\nu(\partial_y h_1)^2\big],\\
&\partial_x u_1+\partial_y u_2=\frac{(1-a)}{Q}h_1\big[(h_1\partial_x+h_2\partial_y)u_1+\nu\partial_y^2h_1\big]-\frac{(1-a)}{Q}( P_t+ P_x u_1)\\
&\qquad+\frac{a}{Q}\big[\kappa\partial_{y}^2\theta+\mu(\partial_y u_1)^2+\nu(\partial_y h_1)^2\big],\\
&\partial_xh_1+\partial_y h_2=0,\\
&(u_1,u_2,\partial_y h_1,h_2)|_{y=0}=\mathbf{0},\quad \theta|_{y=0}=\theta^\ast(t,x),\\
&\lim_{y\rightarrow +\infty}(u_1,\theta,h_1)(t,x,y)=(u_1^e, \theta^e, h_1^e )(t,x,0)=:(U,\Theta,H)(t,x).
\end{aligned}
\right.
\end{equation}
Note that classical solution $(u_1, u_2, \theta, h_1, h_2)$ to the boundary value problem \eqref{simp-lead-noP}, together with the function $\rho$ defined by \eqref{def-rho}, satisfy the original problem \eqref{lead}-\eqref{farfield}.
Hence, we endow the system \eqref{simp-lead-noP} with the initial data
\begin{align}\label{initial}
	(u_1,\theta,h_1)|_{t=0}~=~(u_{1,0},\theta_0,h_{1,0})(x,y),\quad (x,y)\in\mathbb{T}\times\mathbb{R}_+,
\end{align}
and focus on the initial-boundary value problem \eqref{simp-lead-noP}-\eqref{initial} hereafter.

The main result of this paper can be stated as follows.

\begin{thm}\label{original-sol}
For the initial-boundary value problem (\ref{simp-lead-noP})-\eqref{initial} with a smooth outflow $(U,\Theta,H,P)(t,x)$,
assume that the initial data $(u_{1,0},\theta_0,h_{1,0})(x,y)$ and the boundary value $\theta^\ast(t,x)$
are smooth, compatible and  satisfy
\begin{equation}\label{posi}
\theta^\ast(t,x),~\theta_0(x,y),~ h_{1,0}(x,y)\geq 2\delta,
\quad \frac{1}{2}\big(h_{1,0}(x,y)\big)^2 \leq P(0,x)-2\delta,
\end{equation}
for $t\in[0,T],~ (x,y)\in\mathbb{T}\times\mathbb{R}_+$ with some constant $\delta>0$. 
Then there exists $T_\ast\in(0,T]$ and a unique classical solution $(u_1, u_2, \theta,h_1, h_2)$ to the problem (\ref{simp-lead-noP})-\eqref{initial}
in the region $\displaystyle D_{T_\ast}:=\{(t,x,y)|t\in[0,T_\ast],x\in\bT,y\in\bR_+\}$ with the following properties:
\begin{enumerate}
  \item $\theta(t,x,y),~ h_{1}(t,x,y)\geq \delta,\quad \big(h_{1}(x,y)\big)^2/2 \leq P(t,x)-\delta\quad in~D_{T_\ast}$;
  \item $(u_1,\theta,h_1), \partial_y(u_1,\theta,h_1)$
and $\partial_y^2(u_1,\theta,h_1)$ are continuous and bounded in $D_{T_\ast}$,
$\partial_t(u_1,\theta,h_1)$ and $\partial_x(u_1,\theta,h_1)$ are continuous and bounded in any compact set of $D_{T_\ast}$;
  \item $(u_2,h_2)$ and $\partial_y(u_2,h_2)$
are continuous and bounded in any compact set of $D_{T_\ast}$.
\end{enumerate}
\end{thm}
\begin{remark}
The far-field conditions given in \eqref{simp-lead-noP} are compatible with the equations $(\ref{simp-lead-noP})_1$-$(\ref{simp-lead-noP})_3$. Indeed, by restricting the equations in \eqref{euler} on the boundary $\{y=0\}$, and using the boundary conditions \eqref{impermeable} and the fact that $\displaystyle \rho^e(t,x,0)=\frac{P-H^2/2}{R\Theta}(t,x)>0$, we have the following equations by direct calculation, 
\begin{equation*}
\left\{
\begin{aligned}
&U_t +UU_x -\frac{R\Theta H}{P-\frac{1}{2}H^2_1}H_x+\frac{RP_x \Theta}{P-\frac{1}{2}H^2},\\
&\Theta_t+U\Theta_x+\frac{a\Theta H^2}{P+\frac{1}{2}(1-2a)H^2}U_x
-\frac{a(P_t +P_x U)\Theta}{P+\frac{1}{2}(1-2a)H^2}=0,\\
&H_t+UH_x-\frac{(P-\frac{1}{2}H^2)H}{P+\frac{1}{2}(1-2a)H^2}U_x
-\frac{(1-a)(P_t +P_x U)H}{P+\frac{1}{2}(1-2a)H^2}=0.
\end{aligned}
\right.
\end{equation*}
\end{remark}

The rest of the paper is organized as follows. In Section~\ref{transform}, we introduce a nonlinear coordinate transform for the problem \eqref{simp-lead-noP}-\eqref{initial}, and establish the well-posedness theory of the resulting problem. 
The local-in-time existence and uniqueness of the solution to the original problem (\ref{simp-lead-noP})-\eqref{initial} 
will be given in Section~\ref{original}.

\section{Well-posedness of the transformed system}\label{transform}


To show the local well-posedness of the system (\ref{simp-lead-noP})-\eqref{initial}, similar as the classical Prandtl equations, the major difficulty comes from the regularity loss of $x-$derivative. Indeed, we find that from \eqref{simp-lead-noP} the vertical velocity $u_2$, respectively the vertical magnetic component $h_2$, is determined by the equation $\eqref{simp-lead-noP}_4$ and the boundary condition $u_2|_{y=0}=0$, respectively the equation $\eqref{simp-lead-noP}_5$ and the boundary condition $h_2|_{y=0}=0$, which creates a loss of $x-$derivative. The key ingredient of the current work is to develop a new system, which can avoid the regularity loss, from (\ref{simp-lead-noP})-\eqref{initial} by using a nonlinear coordinate transform. Such transformation is
introduced in the work \cite{liu2016mhdboundarylayer}, and based on nondegeneracy assumption on  tangential magnetic component, e.g., $h_1(t,x,y)>0.$ After that, the local well-posedness of the transformed system can be obtained by a proper iteration scheme.


\subsection{Transformation and preliminaries}
As in \cite{liu2016mhdboundarylayer}, by the divergence-free condition
\(
\partial_xh_1+\partial_y h_2=0,
\) and the boundary condition $h_2|_{y=0}=0$, there exists a stream function $\psi=\psi(t,x,y)$ such that
\begin{align}\label{def-psi}
	h_1=\partial_y\psi,\quad h_2=-\partial_x\psi,\quad \psi|_{y=0}=0.
\end{align}
Moreover, by using \eqref{lead} and \eqref{lead-bou}, it follows that $\psi$ satisfies
\begin{equation}\label{stream}
\partial_t\psi+(u_1\partial_x+u_2\partial_y)\psi=\nu\partial_{y}^2\psi.
\end{equation}
Under the assumption of non-zero tangential magnetic component:
\begin{equation}\label{inver-ass}
h_1=\partial_y\psi>0,
\end{equation}
we can introduce the following invertible transformation
\begin{align}\label{coor-trans}
	\tau=t,\quad\xi=x,\quad \eta=\psi(t,x,y),
\end{align}
and new unknown functions
\begin{align}\label{fun-trans}
(\hat{u}_1,\hat{\theta},\hat{h}_1)(\tau,\xi,\eta)~:=~(u_1,\theta,h_1)(t,x,y).
\end{align}

By the coordinate transformation \eqref{coor-trans}, combining with \eqref{def-psi}, \eqref{inver-ass} and the boundedness assumption on $h_1$, the region $\{(t,x,y)| t\in[0,T], x\in\mathbb{T}, y\in\mathbb{R}_+\}$ is changed into the new one 
$$\Omega_T:=[0,T]\times \Omega:=\{(\tau,\xi,\eta)|\tau\in[0,T],\xi\in\bT,\eta\in\bR_+\},$$
  and the boundary $\{y=0\}$ ($y\rightarrow+\infty$, resp.) is changed into $\{\eta=0\}$ ($\eta\rightarrow+\infty$, resp.). Therefore, combining \eqref{fun-trans} and the boundary conditions of \eqref{simp-lead-noP} yields
\begin{align}\label{bd_lead-trans}
	(\hat u_1,\partial_\eta \hat h_1)|_{\eta=0}=\mathbf{0},\quad \hat\theta|_{\eta=0}=\theta^\ast(\tau,\xi),\quad
\lim_{\eta\rightarrow +\infty}(\hat u_1,\hat\theta,\hat h_1)(\tau,\xi,\eta)=(U,\Theta,H)(\tau,\xi).
\end{align}
Then, it implies by combining \eqref{fun-trans} with the initial data \eqref{initial} that
\begin{align}\label{in_lead-trans}
	(\hat u_1,\hat \theta,\hat h_1)|_{\tau=0}=(u_{1,0},\theta_0,h_{1,0})\big(\xi,y(\xi,\eta)\big)=:(\hat u_{1,0},\hat \theta_0,\hat h_{1,0})(\xi,\eta),
\end{align}
where the function $y(\xi,\eta)$ is determined by the following relation:
\begin{align*}
	\displaystyle \eta=\int_0^{y(\xi,\eta)}h_{1,0}(\xi,\zeta)d\zeta.
\end{align*}
Next, with the help of \eqref{stream}, applying \eqref{coor-trans} and \eqref{fun-trans} to the equations in \eqref{simp-lead-noP} yields the corresponding equations for $(\hat{u}_1,\hat{\theta},\hat{h}_1)(\tau,\xi,\eta)$. Thus, we take \eqref{bd_lead-trans} and \eqref{in_lead-trans} into account, and eventually gain the initial-boundary value problem for $(\hat{u}_1,\hat{\theta},\hat{h}_1)(\tau,\xi,\eta)$ in the domain $\Omega_T$ as follows: 
\begin{equation}\label{lead-trans}
\left\{
\begin{aligned}
&\partial_\tau u_1+u_1\partial_\xi u_1-\frac{R\theta}{P-\frac{1}{2}h^2_1}h_1\partial_\xi h_1
+(\nu-\frac{\mu R\theta}{P-\frac{1}{2}h^2_1})h_1\partial_\eta h_1\partial_\eta u_1+\frac{RP_\xi \theta}{P-\frac{1}{2}h^2_1}
=\frac{\mu R\theta h_1^2}{P-\frac{1}{2}h^2_1}\partial_{\eta}^2u_1,\\
&\partial_\tau\theta+\frac{a\theta h_1^2}{Q}\partial_\xi u_1+u_1\partial_\xi\theta
-\frac{\mu a\theta h_1^2(P+\frac{1}{2}h_1^2)}{Q(P-\frac{1}{2}h^2_1)}(\partial_\eta u_1)^2
+\Big(\nu-\frac{\kappa a\theta(P+\frac{1}{2}h_1^2)}{Q(P-\frac{1}{2}h^2_1)}\Big)h_1\partial_\eta h_1\partial_\eta\theta\\
&\quad-\frac{\nu a\theta(P+\frac{1}{2}h^2_1)}{Q(P-\frac{1}{2}h^2_1)}(h_1\partial_\eta h_1)^2-\frac{a(P_\tau+P_\xi u_1)\theta}{Q}=\frac{a\theta h_1^2}{Q}
\Big[\kappa\frac{P+\frac{1}{2}h_1^2}{P-\frac{1}{2}h^2_1}\partial_{\eta}^2\theta-\nu\partial_\eta(h_1\partial_\eta h_1)\Big],\\
&\partial_\tau h_1-\frac{(P-\frac{1}{2}h_1^2)h_1}{Q}\partial_\xi u_1+u_1\partial_\xi h_1
+\frac{\mu ah_1^3}{Q}(\partial_\eta u_1)^2+\frac{\kappa a h_1^2}{Q}\partial_\eta h_1\partial_\eta\theta
+\frac{\nu(P+\frac{1}{2}h^2_1)h_1}{Q}(\partial_\eta h_1)^2\\
&\quad-\frac{(1-a)(P_\tau+P_\xi u_1)h_1}{Q}=\frac{(P-\frac{1}{2}h_1^2)h_1}{Q}
\Big[\nu\partial_\eta(h_1\partial_\eta h_1)-\frac{\kappa ah_1^2}{P-\frac{1}{2}} h_1^2\partial_{\eta}^2\theta\Big],\\
&(u_1,\partial_\eta h_1)|_{\eta=0}=\mathbf{0},\quad \theta|_{\eta=0}=\theta^\ast(\tau,\xi),\quad\lim_{\eta\rightarrow +\infty}(u_1,\theta,h_1)(\tau,\xi,\eta)=(U,\Theta,H)(\tau,\xi),\\
&(u_1,\theta,h_1)|_{\tau=0}=
(u_{1,0},\theta_0,h_{1,0})(\xi,\eta),
\end{aligned}
\right.
\end{equation}
where we have removed all the superscripts 
for simplicity of notations. 
Unless explicitly specified during the rest of this section,  we will replaced $(\hat{u}_1,\hat{\theta},\hat{h}_1)$ by $(u_1,\theta,h_1)$ 
 without any confusion. 
\begin{remark}
The equations in $(\ref{lead-trans})$ are quasi-linear without loss of regularity, so the classical Picard iteration scheme can be used to establish the local existence. Then we can obtain the local well-posedness of solutions to the system \eqref{bd_lead-trans}-(\ref{lead-trans}) in the new coordinates $(\tau,\xi,\eta)$.
However, in order to guarantee the coordinates transformation to be valid, the assumption (\ref{inver-ass}) is required and there is a loss of regularity to transfer the well-posedness results of \eqref{lead-trans} to the ones of original system (\ref{simp-lead-noP})-\eqref{initial}.
\end{remark}

Further, we observe from the problem \eqref{lead-trans} that, it is more convenient to replace $h_1$ by a new unknown
\begin{align}\label{def-q}
 q=q(\tau,\xi,\eta):=h_1^2(\tau,\xi,\eta)/2.
\end{align}
Consequently, we rewrite the system (\ref{lead-trans}) for $\bm{v}=\bm{v}(\tau,\xi,\eta):=(u_1,\theta,q)^{\it T}(\tau,\xi,\eta)$ in the following  form:
\begin{equation}\label{V-sys}
\left\{
\begin{aligned}
&\partial_\tau \bm{v}+\bm{A}(\bm{v})\partial_\xi \bm{v}+\bm{f}(\bm{v},\partial_\eta \bm{v})+\bm{g}(\bm{v})=\bm{B}(\bm{v})\partial_{\eta}^2\bm{v},\\
&(u_1,\partial_\eta q)|_{\eta=0}=\mathbf{0},\quad \theta|_{\eta=0}=\theta^\ast(\tau,\xi),\\
&\lim_{\eta\rightarrow +\infty}\bm{v}(\tau,\xi,\eta)=(U,\Theta,H^2/2)^{\it T}(\tau,\xi)=: \bm{v}_\infty (\tau,\xi),\\
&\bm{v}|_{\tau=0}=\big(u_{1,0},\theta_0,(h_{1,0})^2/2\big)^{\it T}(\xi,\eta)=: \bm{v}_0(\xi,\eta),
\end{aligned}
\right.
\end{equation}
where 
\begin{align}
\displaystyle \bm{A}(\bm{v})&=\left(
\begin{array}{ccc}
u_1 & 0 & -\frac{R\theta}{P-q}\\
\frac{2a\theta q}{Q} & u_1 & 0\\
-\frac{2(P-q)q}{Q} & 0 & u_1\\
\end{array}
\right),~
\bm{B}(\bm{v})=2q\left(
\begin{array}{ccc}
\frac{\mu R\theta}{P-q} & 0 & 0\\
0 & \frac{\kappa a\theta(P+q)}{Q(P-q)} & -\frac{\nu a\theta}{Q}\\
0 & -\frac{2\kappa a q}{Q} &\frac{\nu(P-q)}{Q}\\
\end{array}
\right),\label{AB}
\end{align}
 $\bm{f}(\bm{v},\partial_\eta \bm{v})$ and $\bm{g}(\bm{v})$ are quadratic in relation to $\partial_\eta \bm{v}$ and $\bm{v}$ respectively: 
\begin{align}
\displaystyle \bm{f}(\bm{v},\partial_\eta \bm{v})&=\left(
\begin{array}{ccc}
(\nu-\frac{\mu R\theta}{P-q})\partial_\eta q\partial_\eta u_1\\
\nu\partial_\eta q\partial_\eta\theta
-\frac{a\theta(P+q)}{Q(P-q)}\big[2\mu q(\partial_\eta u_1)^2
+\kappa\partial_\eta q\partial_\eta\theta+\nu(\partial_\eta q)^2\big]\\
\frac{a}{Q}\big[4\mu q^2(\partial_\eta u_1)^2+2\kappa q\partial_\eta q\partial_\eta\theta+\frac{\nu(P+q)}{a}(\partial_\eta q)^2\big]
\end{array}
\right),\label{bF}\\
\displaystyle \bm{g}(\bm{v})&=\Big(\frac{RP_\xi \theta}{P-q},
-\frac{a(P_\tau +P_\xi u_1)\theta}{Q},
-\frac{2(1-a)(P_\tau +P_\xi u_1)q}{Q}\Big)^{\it{T}}.\label{bP}
\end{align}
Moreover, it is worth noting that the above vector functions $\bm{g}(\bm v)$ and $\bm{f}(\bm v,\partial_\eta \bm v)$ given in \eqref{bP} and \eqref{bF} respectively have the following forms:
\begin{align}\label{def_pf}
   \bm{f}(\bm v,\partial_\eta\bm v)=\bm{F}(\bm v,\partial_\eta \bm v)\partial_\eta\bm v,\quad \bm{g}(\bm v)=\bm{G}(\bm v)\bm v,
\end{align}
where the matrices:
\begin{align}\label{F}
\bm{F}(\bm v,\partial_\eta \bm v)&:=\left(
\begin{array}{ccc}
(\nu-\frac{\mu R\theta}{P-q})\partial_\eta q & 0 & 0\\
-\frac{2\mu a\theta q(P+q)}{Q(P-q)}\partial_\eta u_1 &
-\frac{\kappa a\theta(P+q)}{Q(P-q)}\partial_\eta q &
\nu\partial_\eta\theta-\frac{\nu a\theta(P+q)}{Q(P-q)}\partial_\eta q\\
\frac{4\mu a q^2}{Q} \partial_\eta u_1 & \frac{2\kappa aq}{Q} \partial_\eta q &
\frac{\nu(P+q)}{Q}\partial_\eta q
\end{array}
\right),
\end{align}
and 
\begin{align}\label{P}
\bm{G}(\bm v)&:=\left(
\begin{array}{ccc}
0 & \frac{R P_\xi}{P-q} & 0 \\
0 &
-\frac{a(P_\tau+P_\xi u_1)}{Q} &
0\\
0 & 0 &
-\frac{2(1-a)(P_\tau+P_\xi u_1)}{Q}
\end{array}
\right).
\end{align}
The above forms in \eqref{def_pf}, especially the chosen of $\bm F(\bm v,\partial_\eta \bm v)$, play a role in the constructions of approximate solutions to the system \eqref{V-sys} in the next subsection. 

\subsection{Well-posedness results}
In the following, we focus on the initial-boundary value problem \eqref{V-sys} for $\bm{v}$ and investigate its  well-posedness. First of all, we introduce some Sobolev spaces that will be used in the paper. Denote by
\[
\partial^\alpha=\partial^{\alpha_0}_\tau\partial^{\alpha_1}_\xi\partial^{\alpha_2}_\eta, ~\alpha=(\alpha_0, \alpha_1, \alpha_2)\in\mathbb{N}^3,
\]
and we introduce the space $\mathcal{H}^k(\Omega_T)$ for $k\in \mathbb{N}^+$:
\[
\mathcal{H}^k(\Omega_T):=\left\{f(\tau,\xi,\eta):\Omega_T\rightarrow \bR,\
\|f\|_{\mathcal{H}^k(\Omega_T)}:=\sup_{0\leq\tau\leq T}\|f(\tau)\|_{\mathcal{H}^k(\Omega)}<\infty\right\},
\]
where
\[
\|f(\tau)\|_{\mathcal{H}^k(\Omega)}:=\left(\sum_{|\alpha|\leq k}
\|\partial^{\alpha}f(\tau,\cdot)\|^2_{L^2(\Omega)}\right)^\frac{1}{2}.
\]
Also we denote by $H^k(\Omega)$ the Sobolev spaces of measurable function $f\in \Omega_T$ such that for any $\tau\in[0,T]$,
\[
\|f(\tau)\|_{H^k(\Omega)}:=\left(\sum_{\alpha_1+\alpha_2\leq k}
\|\partial^{\alpha_1}_\xi\partial^{\alpha_2}_\eta f(\tau,\cdot)\|^2_{L^2(\Omega)}\right)^\frac{1}{2}<\infty.
\]
Accordingly, we can then define the functions spaces $\mathcal{H}^k(\Omega_T)$ and $H^k(\Omega)$ for vector function $\bm{f}=(f_1,f_2,f_3),\ f_i:\Omega_T\rightarrow\bR,\ i=1,2,3$.

Next, we state the well-posedness result on the transformed problem (\ref{V-sys}) as follows.
\begin{thm}\label{classical-sol}
For the problem \eqref{V-sys} and the integer $k\geq 4$, assume that the known functions $U(\tau,\xi),\Theta(\tau,\xi),H(\tau,\xi), P(\tau,\xi)$ and the boundary value $\theta^\ast(\tau,\xi)$ satisfy $\Theta,H, P, \theta^\ast>0,$ and
\begin{align}\label{outer-bound}
\sup_{0\leq\tau\leq T}\sum_{j=0}^{2k}\|\partial^j_\tau(U,\Theta,H,P,\theta^\ast)(\tau,\cdot)\|_{H^{2k-j}(\bT_\xi)}\leq M_e
\end{align}
for some constant $M_e> 1$. 
Let the initial data $\bm{v}_0(\xi,\eta)=\big(u_{1,0},\theta_0,(h_{1,0})^2/2\big)^{\it T}(\xi,\eta)$
satisfy
\begin{equation}\label{initial-regularity}
\bm{v}_0(\xi,\eta)-\bm{v}_\infty(0,\xi) \in H^{3k}(\Omega),
\end{equation}
and the compatibility conditions up to $k$-th order. 
 Also, we assume there exists a sufficiently small constant $\delta>0$ such that
\begin{equation}\label{initial-lowbound}
\theta_0(\xi,\eta)\geq 2\delta,\quad 2\delta\leq \frac{1}{2}(h_{1,0})^2(\xi,\eta)\leq P(0,\xi)-2\delta, \quad \forall (\xi,\eta)\in\Omega.
\end{equation}
Then there exists a $T_\ast\in(0,T]$ such that the problem (\ref{V-sys}) admits a unique classical solution $\bm{v}(\tau,\xi,\eta)=(u_1,\theta,q)^{\it T}(\tau,\xi,\eta)$ satisfying
$\bm{v}(\tau,\xi,\eta)-\bm{v}_\infty(\tau,\xi) \in \mathcal{H}^k(\Omega_{T_\ast})$. Moreover, 
it holds that
\[
\theta(\tau,\xi,\eta)\geq\delta,\quad \delta\leq q(\tau,\xi,\eta)\leq P(\tau,\xi)-\delta,
\quad \forall (\tau,\xi,\eta)\in\Omega_{T^\ast}.
\]
\end{thm}
\begin{remark}
The regularity assumptions on the outflow $(U,\Theta,H,P)$, the boundary value $\theta_\ast$ and the initial data $\bm{v}_0$ are not optimal. Here, we require high regularity  to simplify the construction of approximate solutions as shown in the following part.
\end{remark}
Recalling the relation \eqref{def-q}, let $\bm{v}(\tau,\xi,\eta)=(u_1,\theta,q)^{\it T}(\tau,\xi,\eta)$ be the classical solution to \eqref{V-sys} given in the above Theorem \ref{classical-sol}, we immediately get that
\[(u_1,\theta,h_1)^{\it T}(\tau,\xi,\eta):=\left(u_1,\theta,\sqrt{2q}\right)^{\it T}(\tau,\xi,\eta)\]
is the classical solution to the problem \eqref{lead-trans}. Consequently, we give the following corollary without proof.
\begin{coro}\label{coro}
	Suppose that in the problem \eqref{lead-trans}, the known functions $U(\tau,\xi),\Theta(\tau,\xi),H(\tau,\xi),$ $P(\tau,\xi)$ and the boundary value $\theta^\ast(\tau,\xi)$  satisfy the same hypotheses as given in Theorem \ref{classical-sol}. 
Let the initial data $(u_{1,0},\theta_0,h_{1,0})^{\it T}(\xi,\eta)$ of the problem (\ref{lead-trans}) 
	satisfy
	\begin{equation*}
	\Big(u_{1,0}(\xi,\eta)-U(0,\xi),\theta_0(\xi,\eta)-\Theta(0,\xi), h_{1,0}(\xi,\eta)-H(0,\xi)\Big)
	\in H^{3k}(\Omega),
	\end{equation*}
	and the compatibility conditions up to $k$-th order. 
	Also, we assume there exists a sufficiently small constant $\delta>0$ such that
	\begin{equation*}
	\theta_0(\xi,\eta), h_{1,0}(\xi,\eta)\geq 2\delta,\quad 
\frac{1}{2}h_{1,0}^2(\xi,\eta) \leq P(0,\xi)-2\delta, \quad \forall (\xi,\eta)\in\Omega.
	\end{equation*}
	Then the problem (\ref{lead-trans}) admits a unique classical solution $(u_1,\theta,h_1)^{\it T}(\tau,\xi,\eta)$ in $[0,T_\ast]$ with $T_\ast$ given in Theorem \ref{classical-sol}. 
Furthermore, it holds that
\[
\theta(\tau,\xi,\eta), h_1(\tau,\xi,\eta)\geq\delta,\quad 
\frac{1}{2}h_1^2(\tau,\xi,\eta)\leq P(\tau,\xi)-\delta,\quad \forall(\tau,\xi,\eta)\in\Omega_{T^\ast}.
\]
\end{coro}

\subsection{Preliminary}
In this subsection, we will make some preliminaries. For the problem \eqref{V-sys},
recall that the facts $\rho, \theta, h_1>0$ implies $P-q>0$ and $Q=P+(1-2a)q>0$ as $0<a<1$. 
We introduce a positive symmetric matrix 
\begin{equation}\label{S}
\displaystyle \bm{S}(\bm{v}):=\left(
\begin{array}{ccc}
\frac{\theta(P-q)}{R} & 0 & 0\\
0 & \frac{P-q}{a} & \theta\\
 0 & \theta & \frac{\theta^2(P+q)}{2q(P-q)}\\
\end{array}
\right),
\end{equation}
such that
\begin{equation}\label{SA}
[\bm{SA}](\bm{v}):=\bm{S}(\bm{v})\bm{A}(\bm{v})=\left(
\begin{array}{ccc}
\frac{\theta (P-q)}{R}u_1 & 0 & -\theta^2\\
0 & \frac{P-q}{a}u_1 & \theta u_1\\
-\theta^2 & \theta u_1 & \frac{\theta^2(P+q)}{2q(P-q)}u_1\\
\end{array}
\right)
\end{equation}
is symmetric, and
\begin{equation}\label{SB}
[\bm{SB}](\bm v):=\bm{S}(\bm{v})\bm{B}(\bm{v})=\left(
\begin{array}{ccc}
2\mu\theta^2 q & 0 & 0\\
0 & 2\kappa\theta q & 0\\
0 & 0 &\nu\theta^2\\
\end{array}
\right)
\end{equation}
is positive definite.

Denote by
\begin{align}\label{initial-j}
  \bm{v}_0^j(\xi,\eta)~:=~\partial_\tau^j\bm{v}(0,\xi,\eta),\quad 0\leq j\leq k.
\end{align}
It follows from the compatibility conditions for \eqref{V-sys} that $\bm{v}_0^j$ can be derived from the equations and initial data of the problem (\ref{V-sys}) by induction with respect to $j$. Precisely, it can be expressed as polynomials of the spatial derivatives, up to order $2j$, of the initial data $\bm{v}_0(\xi,\eta)$. Also, \eqref{initial-regularity} yields that there exists a positive constant $M_0>1$ depending on $\|\bm{v}_0(\xi,\eta)-\bm{v}_\infty(0,\xi)\|_{H^{3k}(\Omega)}$ such that
\begin{align}\label{regular-t}
  \sum_{j=0}^{k}\big\|\bm{v}_0^j(\xi,\eta)-\partial_\tau^j \bm{v}_\infty(0,\xi)\big\|_{H^{3k-2j}(\Omega)}\leq M_0.
\end{align}

Next, in order to overcome the technical difficulty originated from the boundary terms at $\{\eta=0\}$ and $\{\eta=\infty\}$ for the problem \eqref{V-sys}, we introduce an auxiliary function $\phi=\phi(\eta)\in C^\infty(\bR_+)$ such that $0\leq\phi\leq1,$
\begin{equation*}\label{auxiliary}
\phi(\eta)\equiv0\quad \mbox{for}\quad \eta\in[0,1],\qquad\phi(\eta)\equiv1\quad\mbox{for}\quad\eta\geq2.
\end{equation*}
Accordingly, define
\begin{equation}\label{barV}
\bar{\bm v}=\bar {\bm v}(\tau,\xi,\eta):=\Big(\phi U(\tau,\xi),\phi\Theta(\tau,\xi)
+(1-\phi)\theta^\ast(\tau,\xi),H^2(\tau,\xi)/2\Big)^{\it{T}},
\end{equation}
and we use $\bar{\bm v}$ to ensure the homogeneous boundary conditions that will be shown in the sequel. Indeed, let 
\[
\bm{v}-\bar{\bm v}=(\bm v-\bar{\bm v})(\tau,\xi,\eta):=(v_1,\vartheta,w)^{\it{T}}(\tau,\xi,\eta)
\]
with $\bm v$ being the solution to problem \eqref{V-sys}, it then implies
\[
(v_1,\vartheta,\partial_\eta w)|_{\eta=0}=\mathbf{0},\quad
\lim_{\eta\rightarrow +\infty}(v_1,\vartheta,w)(\tau,\xi,\eta)=\mathbf{0}.
\]
Furthermore, it is easy to obtain that by \eqref{outer-bound},
\begin{align}\label{est-V}
\sup_{0\leq\tau\leq T}\sum_{\alpha_0+\alpha_1\leq 2k}\big\|\partial^\alpha \bar{\bm v}(\tau,\cdot)\big\|_{L^{2}(\bT_\xi, L^\infty(\mathbb{R}_\eta^+))}+  \big\|\bar {\bm v}-{\bm v}_\infty\big\|_{\mathcal{H}^{2k}(\Omega_T)}\leq M_1,
\end{align}
and combining \eqref{regular-t} with \eqref{est-V} yields
\begin{align}\label{est_V}
  \sum_{j=0}^{k}\big\|{\bm v}_0^j(\xi,\eta)-\partial_\tau^j \bar{\bm v}(0,\xi,\eta)\big\|_{H^{3k-2j}(\Omega)}\leq M_1,
\end{align}
for some constant $M_1=M_1(M_e, M_0)>1.$

Denote by the notation $[\cdot, \cdot]$ the commutator. We use the notation $A\lesssim B$ to mean that $|A|\leq C |B|$ with a generic constant $C>0$. We also use the following notation for  inner product of vector functions $\bm{f}$ and $\bm{g}:$
\[
\left\langle \bm{f},\bm{g}\right\rangle(\tau)=\left\langle \bm{f}(\tau,\cdot), \bm{g}(\tau,\cdot)\right\rangle_\Omega
:=\sum_{i=1}^3\int_\Omega (f_i g_i)(\tau,\xi,\eta)d\xi d\eta.
\]

We shall use repeatedly the following inequalities, and one can refer to Lemma 2.1 in \cite{liu2016mhdboundarylayer} for the proof.
\begin{lemm}
For proper functions $f, g$ in $\Omega_T$, the following properties hold.

i)  If $\displaystyle \lim\limits_{\eta\rightarrow+\infty}(fg)(\tau,\xi,\eta)=0$, it holds that for $\tau\in[0,T]$,
\begin{equation*}
\left|\int_{\bT}(fg)(\tau,\xi,0) d\xi\right|\leq\|\partial_\eta f(\tau,\cdot)\|_{L^2(\Omega)}\|g(\tau,\cdot)\|_{L^2(\Omega)}+\|\partial_\eta g(\tau,\cdot)\|_{L^2(\Omega)}\|f(\tau,\cdot)\|_{L^2(\Omega)}.
\end{equation*}
In particular, 
\begin{equation}\label{trace}
\|f(\tau,\xi,0)\|_{L^2(\bT_\xi)}\leq \sqrt 2\|f(\tau,\cdot)\|^\frac{1}{2}_{L^2(\Omega)}\|\partial_\eta f(\tau,\cdot)\|^\frac{1}{2}_{L^2(\Omega)},\quad\forall f(\tau,\cdot)\in H^1(\Omega) .
\end{equation}

ii) For $f, g\in \mathcal{H}^k(\Omega_T)$ with some integer $k\geq 4$, it holds that for $|\alpha|+|\beta|\leq k$ and $t\leq T,$
\begin{equation}\label{moser0}
\big\|(\partial^\alpha f\partial^\beta g)(\tau,\cdot) \big\|_{L^2(\Omega)}\lesssim\|f(\tau)\|_{\mathcal{H}^k(\Omega)}\cdot\|g(\tau)\|_{\mathcal{H}^k(\Omega)}.
\end{equation}
Then, it implies that for $|\alpha|\leq k$,
\begin{align}\label{moser1}
\begin{aligned}
&\big\|\partial^\alpha(fg)(\tau,\cdot)\big\|_{L^2(\Omega)}\lesssim \|f(\tau)\|_{\mathcal{H}^k(\Omega)}\cdot\|g(\tau)\|_{\mathcal{H}^k(\Omega)},
\end{aligned}\end{align}
and 
\begin{align}\label{moser2}
\begin{aligned}
&\big\|\big([\partial^\alpha, f]~g\big)(\tau,\cdot)\big\|_{L^2(\Omega)} \lesssim \|f(\tau)\|_{\mathcal{H}^k(\Omega)}\cdot\|g(\tau)\|_{\mathcal{H}^{k-1}(\Omega)}.
\end{aligned}\end{align}

\end{lemm}

\subsection{The iteration scheme}
We are going to use the classical iteration scheme to prove the local existence result in Theorem~\ref{classical-sol}. To this end, we  construct a sequence of approximate solutions $\{\bm{v}^n\}_{n\geq 0}$ as follows.

\indent\newline
\underline{\textit{The zero-th order approximate solution:}} 
We want to construct the zero-th order approximate solution $\bm{v}^0(\tau,\xi,\eta)$ of problem \eqref{V-sys}, such that $\bm{v}^0=(u_1^0,\theta^0, q^0)^{\it{T}}$ satisfies the boundary conditions of \eqref{V-sys}, i.e.,
\begin{align*}
  (u_1^0, \partial_\eta q^0)|_{\eta=0}=\mathbf{0},\quad \theta^0|_{\eta=0}=\theta^{\ast}(\tau,\xi),\qquad \lim\limits_{\eta\rightarrow+\infty}\bm{v}^0(\tau,\xi,\eta)=\bm{v}_\infty(\tau,\xi),
\end{align*}
and the compatibility conditions
\begin{align*}
  \partial_\tau^j {\bm v}^0|_{\tau=0}=\bm{v}_0^j(\xi,\eta),\quad 0\leq j\leq k.
\end{align*}

With the help of \eqref{initial-j} and \eqref{barV} the zero-th order approximate solution ${\bm v}^0(\tau,\xi,\eta)$ of \eqref{V-sys} can be chosen as follows:
\begin{equation}\label{appro}
{\bm v}^0(\tau,\xi,\eta)=\bar{\bm v}(\tau,\xi,\eta)+\sum_{j=0}^{k}\frac{\tau^j}{j!}\Big({\bm v}_0^j(\xi,\eta)-\partial_\tau^j\bar{\bm v}(0,\xi,\eta)\Big).
\end{equation}
Then for this approximation $\bm v^0$, we have the following proposition.
\begin{prop}\label{zero-appro}
The zeroth-order approximate solution $\bm v^0(\tau,\xi,\eta)=(u_1^0,\theta^0,q^0)^{\it T}(\tau,\xi,\eta)$ of the problem (\ref{V-sys}), defined by \eqref{appro}, satisfies
\begin{align}\label{est-V0}
  \bm v^0(\tau,\xi,\eta)-\bar{\bm v}(\tau,\xi,\eta)\in\mathcal{H}^k(\Omega_{T}),
\end{align}
\begin{equation}\label{appro-init}
\partial^j_\tau \bm v^0|_{\tau=0}=\bm v^j_0 (\xi,\eta),\quad 0\leq j\leq k,
\end{equation}
and the boundary conditions
\begin{equation}\label{appro-bou}
(u_1^0,\partial_\eta q^0)|_{\eta=0}=\mathbf{0},\quad \theta^0|_{\eta=0}=\theta^\ast(\tau,\xi),\quad \lim_{\eta\rightarrow +\infty}\bm v^0(\tau,\xi,\eta)=\bm v_\infty(\tau,\xi).
\end{equation}
Moreover, 
there exists some $T_0\in(0,T]$ such that for $(\tau,\xi,\zeta)\in\Omega_{T_0},$ 
\begin{equation}\label{appro-posi}
\theta^0(\tau,\xi,\zeta)\geq \delta,\quad \delta\leq q^0(\tau,\xi,\zeta)\leq P(\tau,\xi)-\delta
\end{equation}
with $\delta>0$ given in (\ref{initial-lowbound}).
\end{prop}
\begin{proof}
From the definition \eqref{appro} of $\bm v^0$, direct calculation shows that  $\bm v^0$ satisfies (\ref{appro-init}) and the boundary conditions (\ref{appro-bou}) by virtue of \eqref{barV} and the compatibility conditions.
Combining \eqref{appro} with \eqref{est_V}, one can obtain that 
\begin{align}\label{est_V0}
\|\bm v^0-\bar{\bm v}\|_{\mathcal{H}^k(\Omega_T)}^2\leq&\sum^k_{j=0}
\left\|\frac{\tau^j}{j!}\big(\bm {\bm v}_0^j(\xi,\eta)-\partial^j_\tau\bar{\bm v}(0,\xi,\eta)\big)\right\|_{\mathcal{H}^k(\Omega_T)}^2
\lesssim (1+T^k)M_1
,
\end{align}
which implies \eqref{est-V0} immediately.

Next, from \eqref{appro} it holds that
\begin{align}\label{positive-theta0}
\begin{aligned}
\theta^0(\tau,\xi,\eta)&=\theta_0(\xi,\eta)+\phi(\eta)\big(\Theta(\tau,\xi)-\Theta(0,\xi)\big)+(1-\phi(\eta))\big(\theta^\ast(\tau,\xi)-\theta^\ast(0,\xi)\big)\\
&\quad+\sum^k_{j=1}\frac{\tau^j}{j!}\partial^j_\tau\big((\theta-\Theta)
+(1-\phi)(\Theta-\theta^\ast)\big)(0,\xi,\eta)\\
&\geq \theta_0(\xi,\eta)-\tau\sup_{0\leq\tau\leq T}\|\partial_\tau(\Theta,\theta^\ast)(\tau,\cdot)\|_{L^\infty(\bT_\xi)}\\
&\quad-\sum^k_{j=1}\frac{\tau^j}{j!}\Big(\|\partial^j_\tau(\theta-\Theta)|_{\tau=0}\|_{L^\infty(\Omega)}
+\|\partial^j_\tau(\Theta-\theta^\ast)(0,\xi)\|_{L^\infty(\bT_\xi)}\Big),
	\end{aligned}
\end{align}
where we have used the fact that
\[
|f(\tau,\xi)-f(0,\xi)|\leq\tau\sup_{0\leq\tau\leq T}\|\partial_\tau f(\tau,\cdot)\|_{L^\infty(\bT_\xi)},\quad
\forall ~(\tau,\xi)\in[0,T]\times\bT.
\]
Then, applying the Sobolev embedding inequalities to \eqref{positive-theta0} and using \eqref{regular-t} and \eqref{est-V}, it yields that for any $\tau\in[0,T],$
\begin{align*}
\theta^0(\tau,\xi,\eta)
&\geq \theta_0(\xi,\eta) -C\tau(1+\tau^{k-1})\sum^k_{j=1}\Big(\|\partial^j_\tau(\theta-\Theta)|_{\tau=0}\|_{H^2(\Omega)}
+\sup_{0\leq\tau\leq T_0}\|\partial^j_\tau(\Theta,\theta^\ast)\|_{H^1(\bT_\xi)}\Big)\nonumber\\
&\geq \theta_0(\xi,\eta)-C\tau(1+T^{k-1})(\sqrt{M_0}+\sqrt{M_e}).
\end{align*}
Similarly, one can obtain
\begin{align*}
q^0(\tau,\xi,\eta)&\geq \frac{1}{2}(h_{1,0})^2(\xi,\eta) -C\tau(1+T^{k-1})(\sqrt{M_0}+\sqrt{M_e}),
\end{align*}
and
\begin{align*}
P(\tau,\xi)-q^0(\tau,\xi,\eta)&\geq P(0,\xi)-\frac{1}{2}(h_{1,0})^2(\xi,\eta) -C\tau(1+T^{k-1})(\sqrt{M_0}+\sqrt{M_e}).
\end{align*}
Thus, recalling the assumption (\ref{initial-lowbound}) and choosing 
\begin{align*}
T_0=\min\{1,\tilde{T}_0\}\quad\mbox{with}\quad\tilde{T}_0=\frac{\delta}{C(1+T^{k-1})(\sqrt{M_0}+\sqrt{M_e})},
\end{align*}
the statement 
(\ref{appro-posi}) follows immediately.
\end{proof}

\smallskip
\indent\newline
\underline{\textit{The $n-$th order approximation:}}
Assume that in the region $\Omega_{T_1},\ T_1\in(0,T_0]$ with $T_0$ given in Proposition \ref{zero-appro} and $T_1$ to be determined later,
the $(n-1)$-th order approximate solution $\bm {\bm v}^{n-1}(\tau,\xi,\eta)=(u_1^{n-1},\theta^{n-1},q^{n-1})^{\it{T}}(\tau,\xi,\eta), n\geq 1$ to the problem (\ref{V-sys}) has been constructed. Also,  ${\bm v}^{n-1}$ satisfies
\begin{align}\label{est_n-1}
\bm {\bm v}^{n-1}(\tau,\xi,\eta)-\bar{\bm v}(\tau,\xi,\eta)\in \mathcal{H}^k(\Omega_{T_1}),\quad k\geq 4,
\end{align}
  and  the following induction hypotheses
\begin{equation}\label{induc-hypo}
\left\{\begin{aligned}
&\partial^j_\tau \bm {\bm v}^{n-1}|_{\tau=0}=\bm {\bm v}_0^j(\xi,\eta),\quad 0\leq j\leq k,\\
&(u_1^{n-1},\partial_\eta q^{n-1})|_{\eta=0}=\mathbf{0},\quad \theta^{n-1}|_{\eta=0}=\theta^\ast(\tau,\xi),\quad
\lim_{\eta\rightarrow +\infty}\bm {\bm v}^{n-1}(\tau,\xi,\eta)=\bm v_\infty(\tau,\xi),\\
&\theta^{n-1}(\tau,\xi,\eta)\geq\delta,\quad \delta\leq q^{n-1}(\tau,\xi,\eta)\leq P(\tau,\xi)-\delta.
\end{aligned}
\right.
\end{equation}
Now, we are going to construct the $n$-th order approximate solution
$$\bm {\bm v}^n=\bm {\bm v}^n(\tau,\xi,\eta):=(u_1^n,\theta^n,q^n)^{\it T}(\tau,\xi,\eta)$$
 to the problem (\ref{V-sys}) that satisfies the corresponding conditions to \eqref{est_n-1}- \eqref{induc-hypo}. 
Actually, we construct 
$\bm {\bm v}^n=(u_1^n,\theta^n, q^n)^{\it T}$ 
by solving the following linear initial-boundary value problem in $\Omega_{T_1}$,
\begin{equation}\label{picard}
\left\{
\begin{aligned}
&\partial_\tau \bm {\bm v}^n +\bm{A}(\bm {\bm v}^{n-1})\partial_\xi \bm {\bm v}^n
+\bm{F}(\bm {\bm v}^{n-1},\partial_\eta \bm {\bm v}^{n-1})\partial_\eta \bm {\bm v}^n+\bm{G}(\bm {\bm v}^{n-1})\bm {\bm v}^n=\bm{B}(\bm {\bm v}^{n-1})\partial_{\eta}^2\bm {\bm v}^n,\\
&(u_1^n,\partial_\eta q^n)|_{\eta=0}=\mathbf{0},\quad \theta^n|_{\eta=0}=\theta^\ast(\tau,\xi),\quad
\lim_{\eta\rightarrow +\infty}\bm {\bm v}^n(\tau,\xi,\eta)=\bm v_\infty(\tau,\xi),\\
&\bm {\bm v}^n|_{\tau=0}=\bm {\bm v}_0(\xi,\eta),
\end{aligned}
\right.
\end{equation}
where $ \bm{A}(\cdot), \bm{B}(\cdot),$ are defined in (\ref{AB}), and $ \bm{G}(\cdot), \bm{F}(\cdot,\cdot)$ are defined in (\ref{P}), \eqref{F} respectively.

Firstly, we will show the solvability of the above problem \eqref{picard}.
\begin{lemm}\label{energy-est}
Under the assumptions of Theorem \ref{classical-sol}, the problem \eqref{picard} admits a unique classical solution $\bm v^n(\tau,\xi,\eta)$ in $\Omega_{T_1}$ satisfying
\begin{align*}
	\bm {\bm v}^{n}(\tau,\xi,\eta)-\bar{\bm v}(\tau,\xi,\eta)\in \mathcal{H}^k(\Omega_{T_1}),\quad k\geq 4,
\end{align*}
and 
\begin{align}\label{n-initial}
  \partial^j_\tau \bm v^{n}|_{\tau=0}=\bm {\bm v}_0^j(\xi,\eta),\quad 0\leq j\leq k.
\end{align}
Moreover, there exists a positive constant $C_0$, depending on $T_0, M_e, \delta,$ 
such that 
the following estimate holds:
\begin{align}\label{gronwall-general}
\begin{aligned}
&\|\bm {\bm v}^n-\bar{\bm v}\|^2_{\mathcal{H}^k(\Omega_t)}
+\int^t_0\big\|\partial_\eta( \bm {\bm v}^n-\bar{\bm v})(\tau)\big\|^2_{\mathcal{H}^k(\Omega)}d\tau\\
&\leq C_0M_1^3 ~\exp\left\{C_0\int_{0}^{t}\big\|(\bm {\bm v}^{n-1}-\bar{\bm v})(\tau)\big\|_{\mathcal{H}^k(\Omega)}^{18}d\tau\right\},
\quad\forall t\in[0,T_1],
\end{aligned}\end{align}
where the constant $M_1\geq 1$ is given in \eqref{est_V}.
\end{lemm}

\begin{proof}
We are mainly devoted to proving the estimate \eqref{gronwall-general}, since
the well-posedness of the linear problem \eqref{picard} is standard. Indeed, one can use the a priori estimate \eqref{gronwall-general} and a compactness argument  to show the local existence and uniqueness of a solution to problem \eqref{picard}.  The lifetime $T_1$ of the solution can be obtained by again using \eqref{gronwall-general} and a continuous induction argument. Furthermore, from the definition \eqref{initial-j} of $\bar v_0^j$ and the induction hypothesis $\eqref{induc-hypo}_1$ on $\bm v^{n-1}$, direct calculation shows \eqref{n-initial} immediately. 
 
Recall the known function $\bar{\bm v}$ defined by \eqref{barV} and set
$$\bm {\bm v}^n-\bar{\bm v}:=(v_1^n,\vartheta^n, w^n)^{\it  T}$$
to homogenize the boundary conditions of $\bm v^n$. 
 We know $\bm {\bm v}^n-\bar{\bm v}$ satisfies the following initial-boundary value problem,
\begin{equation}\label{lead-iden}
\begin{cases}
\partial_\tau(\bm {\bm v}^n-\bar{\bm v})+\bm{A}({\bm v}^{n-1})\partial_\xi({\bm v}^n-\bar{\bm v})
+\bm{F}({\bm v}^{n-1},\partial_\eta {\bm v}^{n-1})\partial_\eta {\bm v}^n\\
\qquad+\bm{G}({\bm v}^{n-1})({\bm v}^n-\bar{\bm v})-\bm{B}({\bm v}^{n-1})\partial_{\eta}^2({\bm v}^n-\bar{\bm v})=\bm{r}^{n-1},\\
(v_1^n,\vartheta^n, \partial_\eta w^n)|_{\eta=0}=\mathbf{0},\quad
\lim\limits_{\eta\rightarrow +\infty}({\bm v}^n-\bar{\bm v})(\tau,\xi,\eta)=\mathbf{0},\\
({\bm v}^n-\bar{\bm v})|_{\tau=0}={\bm v}_0(\xi,\eta)-\bar{\bm v}(0,\xi,\eta),
\end{cases}
\end{equation}
where the error term
\begin{align}\label{def-G}
  \bm{r}^{n-1}:=-\partial_\tau \bar{\bm v}-\bm{A}({\bm v}^{n-1})\partial_\xi\bar{\bm v}
  -\bm{G}({\bm v}^{n-1})\bar{\bm v}+\bm{B}({\bm v}^{n-1})\partial_{\eta}^2\bar{\bm v}.
\end{align}
Applying the operator $\partial^\alpha=\partial^{\alpha_0}_\tau\partial^{\alpha_1}_\xi\partial^{\alpha_2}_\eta, \ |\alpha|\leq k$, on the equations of $(\ref{lead-iden})_1$, it yields that
\begin{align}\label{partial-iden}
\begin{aligned}
&\partial_\tau\partial^\alpha({\bm v}^n-\bar{\bm v})+\bm{A}({\bm v}^{n-1})\partial_\xi\partial^\alpha({\bm v}^n-\bar{\bm v})
+\bm{F}({\bm v}^{n-1},\partial_\eta {\bm v}^{n-1})\partial_\eta\partial^\alpha {\bm v}^n\\
&\quad+\partial^\alpha\big(\bm{G}({\bm v}^{n-1})({\bm v}^n-\bar{\bm v})\big)-\bm{B}({\bm v}^{n-1})\partial_{\eta}^2\partial^\alpha ({\bm v}^n-\bar{\bm v})\\
&=-\big[\partial^\alpha,\bm{F}({\bm v}^{n-1},\partial_\eta {\bm v}^{n-1})\big]\partial_\eta ({\bm v}^n-\bar{\bm v})-\big[\partial^\alpha,\bm{A}({\bm v}^{n-1})\big]\partial_\xi({\bm v}^n-\bar{\bm v})\\
&\quad
+\big[\partial^\alpha,\bm{B}({\bm v}^{n-1})\big]\partial_{\eta}^2({\bm v}^n-\bar{\bm v})+\partial^\alpha\bm{r}^{n-1}.
\end{aligned}\end{align}
We now take the inner product of $\partial^\alpha({\bm v}^n-\bar{\bm v})^{\it T}\bm{S}({\bm v}^{n-1})$ and the identity (\ref{partial-iden}) over $\Omega=\{(\xi,\eta)|\xi\in\bT,\eta\in\bR_+\}$, where the matrix $\bm S(\cdot)$ is given in \eqref{S}, then estimate the resulting equation term by term. In this proof, denote by $C>0$ the generic constant depending only on $T_0, M_e, M_1,  \delta$ and the parameters of \eqref{lead-trans}, which may be different from line to line. 

\textbf{Step 1.} In this step, we are going to estimate the terms on the left hand side of \eqref{partial-iden}. Firstly, since $\bm S(\cdot)$ is symmetric, one has
\begin{align}\label{t-est0}
\begin{aligned}
&\left\langle \partial^\alpha({\bm v}^n-\bar{\bm v})^{\it T}\bm{S}({\bm v}^{n-1}),
\partial_\tau \partial^\alpha({\bm v}^n-\bar{\bm v})\right\rangle(\tau)\\
&=\frac{1}{2}\frac{d}{d\tau}\left\langle \partial^\alpha({\bm v}^n-\bar{\bm v})^{\it T}\bm{S}({\bm v}^{n-1}),
\partial^\alpha({\bm v}^n-\bar{\bm v})\right\rangle(\tau)\\
&\quad-\frac{1}{2}\left\langle\partial^\alpha({\bm v}^n-\bar{\bm v})^{\it T}\partial_\tau\bm{S}({\bm v}^{n-1}),
\partial^\alpha({\bm v}^n-\bar{\bm v})\right\rangle(\tau).
\end{aligned}\end{align}
It is easy to get that
\begin{align}\label{t-est}
\begin{aligned}
&\left|\left\langle\partial^\alpha({\bm v}^n-\bar{\bm v})^{\it T}\partial_\tau\bm{S}({\bm v}^{n-1}),
\partial^\alpha({\bm v}^n-\bar{\bm v})\right\rangle(\tau)\right|\\
&\leq\|\partial_\tau \bm{S}({\bm v}^{n-1})(\tau,\cdot)\|_{L^\infty(\Omega)}\cdot
\|\partial^\alpha({\bm v}^n-\bar{\bm v})(\tau,\cdot)\|^2_{L^2(\Omega)}\\
&\lesssim \big(1+\|({\bm v}^{n-1}-\bar{\bm v})(\tau)\|_{\mathcal{H}^3(\Omega)}^4\big)
\big\|({\bm v}^n-\bar{\bm v})(\tau)\big\|^2_{\mathcal{H}^k(\Omega)},
\end{aligned}\end{align}
where we have used the Sobolev embedding inequality in the last inequality. Note that in the rest of the proof, such kind of embedding inequalities will be used repeatedly, and we won't specify them one by one.  Plugging \eqref{t-est} into \eqref{t-est0} gives that
\begin{align}\label{S-est}
\begin{aligned}
&\left\langle\partial^\alpha({\bm v}^n-\bar{\bm v})^{\it T}\bm{S}({\bm v}^{n-1}),
\partial_\tau \partial^\alpha({\bm v}^n-\bar{\bm v})\right\rangle(\tau)\\
&\geq\frac{1}{2}\frac{d}{d\tau}\left\langle\partial^\alpha({\bm v}^n-\bar{\bm v})^{\it T}\bm{S}({\bm v}^{n-1}),
\partial^\alpha({\bm v}^n-\bar{\bm v})\right\rangle(\tau)\\
&\quad-C\big(1+\|({\bm v}^{n-1}-\bar{\bm v})(\tau)\|_{\mathcal{H}^3(\Omega)}^4\big)
\big\|({\bm v}^n-\bar{\bm v})(\tau)\big\|^2_{\mathcal{H}^k(\Omega)}.
\end{aligned}\end{align}

Secondly, recalling the symmetric matrix $[\bm{SA}](\cdot)$ in (\ref{SA}), one has that by integration by parts, 
\begin{align}\label{A-est}
\begin{aligned}
&\left|\left\langle\partial^\alpha({\bm v}^n-\bar{\bm v})^{\it T}\bm{S}({\bm v}^{n-1}),
\bm{A}({\bm v}^{n-1})\partial_\xi\partial^\alpha({\bm v}^n-\bar{\bm v})\right\rangle(\tau)\right|\\
&=\frac{1}{2}\left|\left\langle\partial^\alpha({\bm v}^n-\bar{\bm v})^{\it T}\partial_\xi\big([\bm{SA}]({\bm v}^{n-1})\big),
\partial^\alpha({\bm v}^n-\bar{\bm v})\right\rangle(\tau)\right|\\
&\leq\frac{1}{2}\left\|\partial_{\xi}\big([\bm{SA}]({\bm v}^{n-1})\big)(\tau,\cdot)\right\|_{L^\infty(\Omega)}\|\partial^\alpha({\bm v}^n-\bar{\bm v})(\tau,\cdot)\|^2_{L^2(\Omega)}\\
&\lesssim \big(1+\|({\bm v}^{n-1}-\bar{\bm v})(\tau)\|_{H^3(\Omega)}^5\big)\big\|({\bm v}^n-\bar{\bm v})(\tau)\big\|^2_{\mathcal{H}^k(\Omega)}.
\end{aligned}\end{align}

Thirdly, note that
\begin{align}\label{SF}
\begin{aligned}
&[\bm{SF}](\bm v,\partial_\eta \bm v):=\bm{S}(\bm v)\bm{F}(\bm v,\partial_\eta \bm v)\\
&\quad=\left(
\begin{array}{ccc}
(\frac{\nu(P-q)}{R}-\mu\theta)\theta\partial_\eta q & 0 & 0\\
-2\mu\theta q\partial_\eta u_1 & -\kappa\theta\partial_\eta q & \frac{\nu(P-q)}{a}\partial_\eta\theta\\
0 & 0 & \nu\theta\big(\partial_\eta\theta+\frac{\theta(P+q)}{2q(P-q)}\partial_\eta q\big)
\end{array}\right).
\end{aligned}\end{align}
Direct calculation yields that
\begin{align}\label{F-est}
\begin{aligned}
&\big|\left\langle\partial^\alpha({\bm v}^n-\bar{\bm v})^{T}\bm{S}({\bm v}^{n-1}),
\bm{F}({\bm v}^{n-1},\partial_\eta {\bm v}^{n-1})\partial_\eta\partial^\alpha {\bm v}^n\right\rangle(\tau)\big|\\
&\leq\left\|[\bm{SF}]({\bm v}^{n-1},\partial_\eta {\bm v}^{n-1})(\tau,\cdot)\right\|_{L^\infty(\Omega)}\|\partial^\alpha({\bm v}^n-\bar{\bm v})(\tau,\cdot)\|_{L^2(\Omega)}
\|\partial_\eta\partial^\alpha {\bm v}^n(\tau)\|_{L^2(\Omega)}\\
&\lesssim\big(1+\|({\bm v}^{n-1}-\bar{\bm v})(\tau)\|_{H^3(\Omega)}^4\big)\big\|({\bm v}^n-\bar{\bm v})(\tau)\big\|^2_{\mathcal{H}^k(\Omega)}
\cdot\left(1+\big\|\partial_\eta({\bm v}^n-\bar{\bm v})(\tau)\big\|^2_{\mathcal{H}^k(\Omega)}\right).
\end{aligned}\end{align}
Also, it follows that
\begin{align}\label{P-est}
\begin{aligned}
  &\left|\left\langle\partial^\alpha({\bm v}^n-\bar{\bm v})^{\it T}\bm{S}({\bm v}^{n-1}),
  \partial^\alpha\big(\bm{G}({\bm v}^{n-1}) ({\bm v}^n-\bar{\bm v})\big)\right\rangle(\tau)\right|\\
  &\leq\|\bm{S}({\bm v}^{n-1})(\tau,\cdot)\|_{L^\infty(\Omega)}\|\partial^\alpha({\bm v}^n-\bar{\bm v})(\tau,\cdot)\|_{L^2(\Omega)}
  \|\partial^\alpha\big(\bm{G}({\bm v}^{n-1})({\bm v}^n-\bar{\bm v})\big)(\tau,\cdot)\|_{L^2(\Omega)}\\
  &\lesssim \big(1+\|({\bm v}^{n-1}-\bar{\bm v})(\tau)\|_{H^2(\Omega)}^3\big)\big\|({\bm v}^n-\bar{\bm v})(\tau)\big\|_{\mathcal{H}^k(\Omega)}\\
  &\quad\cdot\big(1+\|({\bm v}^{n-1}-\bar{\bm v})(\tau)\|_{\mathcal{H}^k(\Omega)}\big)\big\|({\bm v}^n-\bar{\bm v})(\tau)\big\|_{\mathcal{H}^k(\Omega)}\\
  &\lesssim \big(1+\|({\bm v}^{n-1}-\bar{\bm v})(\tau)\|_{\mathcal{H}^k(\Omega)}^4\big)\big\|({\bm v}^n-\bar{\bm v})(\tau)\big\|_{\mathcal{H}^k(\Omega)}^2,
\end{aligned}\end{align}
where we have used from \eqref{moser1} and the definition \eqref{P} of $\bm G,$
\begin{align*}
  &\big\|\partial^\alpha\big(\bm{G}({\bm v}^{n-1})({\bm v}^n-\bar{\bm v})\big)(\tau,\cdot)\big\|_{L^2(\Omega)}
  \lesssim\big(1+\|({\bm v}^{n-1}-\bar{\bm v})(\tau)\|_{\mathcal{H}^k(\Omega)}\big)\|({\bm v}^n-\bar{\bm v})(\tau)\|_{\mathcal{H}^k(\Omega)}.
\end{align*}

Fourthly, we estimate the last term on the left side of \eqref{partial-iden}. 
Since the far-field states vanish as given in $(\ref{lead-iden})$, 
we obtain that by integration by parts,
\begin{align*}
&-\left\langle\partial^\alpha({\bm v}^n-\bar{\bm v})^{\it T}\bm{S}({\bm v}^{n-1}),
\bm{B}({\bm v}^{n-1})\partial_{\eta}^2\partial^\alpha ({\bm v}^n-\bar{\bm v})\right\rangle(\tau)\nonumber\\
&=\int_\bT\partial^\alpha({\bm v}^n-\bar{\bm v})^{\it T}[\bm{SB}]({\bm v}^{n-1})\partial_\eta\partial^\alpha({\bm v}^n-\bar{\bm v})\big|_{\eta=0}d\xi\nonumber\\
&\quad+\left\langle\partial_\eta\partial^\alpha({\bm v}^n-\bar{\bm v})^{\it T}[\bm{SB}]({\bm v}^{n-1}),
\partial_{\eta}\partial^\alpha ({\bm v}^n-\bar{\bm v})\right\rangle(\tau)\nonumber\\
&\quad+\left\langle\partial^\alpha({\bm v}^n-\bar{\bm v})^{\it T}\partial_\eta\big([\bm{SB}]({\bm v}^{n-1})\big),
\partial_{\eta}\partial^\alpha ({\bm v}^n-\bar{\bm v})\right\rangle(\tau).
\end{align*}
Then along with that the matrix $[\bm{SB}](\cdot)$ given in (\ref{SB}) is positive definite, gives that
\begin{align}\label{B_est}
\begin{aligned}
&-\left\langle\partial^\alpha({\bm v}^n-\bar{\bm v})^{\it T}\bm{S}({\bm v}^{n-1}),
\bm{B}({\bm v}^{n-1})\partial_{\eta}^2\partial^\alpha ({\bm v}^n-\bar{\bm v})\right\rangle(\tau)\\
&\geq \int_\bT\partial^\alpha({\bm v}^n-\bar{\bm v})^{\it T}[\bm{SB}]({\bm v}^{n-1})
\partial_\eta\partial^\alpha({\bm v}^n-\bar{\bm v})\big|_{\eta=0}d\xi+ c_\delta\|\partial_\eta\partial^\alpha({\bm v}^n-\bar{\bm v})(\tau,\cdot)\|^2_{L^2(\Omega)}\\
&\quad-\big\|\partial_\eta\big([\bm{SB}]({\bm v}^{n-1})\big)(\tau,\cdot)\big\|_{L^\infty(\Omega)}
\|\partial^\alpha({\bm v}^n-\bar{\bm v})(\tau,\cdot)\|_{L^2(\Omega)}\|\partial_\eta\partial^\alpha({\bm v}^n-\bar{\bm v})(\tau,\cdot)\|_{L^2(\Omega)}\\
&\geq c_\delta\|\partial_\eta\partial^\alpha({\bm v}^n-\bar{\bm v})(\tau,\cdot)\|^2_{L^2(\Omega)}
+\int_\bT\partial^\alpha({\bm v}^n-\bar{\bm v})^{\it T}[\bm{SB}]({\bm v}^{n-1})
\partial_\eta\partial^\alpha({\bm v}^n-\bar{\bm v})\big|_{\eta=0}d\xi\\
&\quad-C\big(1+\|({\bm v}^{n-1}-\bar{\bm v})(\tau)\|_{H^3(\Omega)}^5\big)
\big\|({\bm v}^n-\bar{\bm v})(\tau)\big\|_{\mathcal{H}^k(\Omega)}\big\|\partial_\eta({\bm v}^n-\bar{\bm v})(\tau)\big\|_{\mathcal{H}^k(\Omega)},
\end{aligned}\end{align}
where and in the sequel, $c_\delta=c(\delta)>0$ depending only on $\delta$ and the outflow.
To estimate the second term on the right-hand side of (\ref{B_est}), we use the following claim.
\begin{claim}\label{claim1}
\begin{align*}
&\Big|\int_\bT\partial^\alpha({\bm v}^n-\bar{\bm v})^{\it T}
[\bm{SB}]({\bm v}^{n-1})\partial_\eta\partial^\alpha({\bm v}^n-\bar{\bm v})\big|_{\eta=0}d\xi\nonumber\\
&\quad-\int_\bT\partial^\alpha({\bm v}^n-\bar{\bm v})^{\it T}
[\bm{SF}]({\bm v}^{n-1},\partial^\alpha {\bm v}^{n-1})\partial_\eta {\bm v}^n\big|_{\eta=0}d\xi\Big|\nonumber\\
&\lesssim \big(1+\|({\bm v}^{n-1}-\bar{\bm v})(\tau)\|_{\mathcal{H}^k(\Omega)}^{9}\big)\big(1+\|({\bm v}^n-\bar{\bm v})(\tau)\|_{\mathcal{H}^k(\Omega)}\big)
\big\|\partial_\eta({\bm v}^n-\bar{\bm v})(\tau)\big\|_{\mathcal{H}^k(\Omega)}.
\end{align*}
\end{claim}
For the moment, we assume Claim \ref{claim1} to be proved later, such that it follows from \eqref{B-est} that
\begin{align}\label{B-est}
\begin{aligned}
&-\left\langle\partial^\alpha({\bm v}^n-\bar{\bm v})^{\it T}\bm{S}({\bm v}^{n-1}),
\bm{B}({\bm v}^{n-1})\partial_{\eta}^2\partial^\alpha ({\bm v}^n-\bar{\bm v})\right\rangle(\tau)\\
&\geq c_\delta\|\partial_\eta\partial^\alpha({\bm v}^n-\bar{\bm v})(\tau,\cdot)\|^2_{L^2(\Omega)}+\int_\bT\partial^\alpha({\bm v}^n-\bar{\bm v})^{\it T}[\bm{SF}]({\bm v}^{n-1},\partial^\alpha {\bm v}^{n-1})\partial_\eta {\bm v}^n\big|_{\eta=0}d\xi\\
&\quad -C\big(1+\|({\bm v}^{n-1}-\bar{\bm v})(\tau)\|_{\mathcal{H}^k(\Omega)}^{9}\big)\big(1+\|({\bm v}^n-\bar{\bm v})(\tau)\|_{\mathcal{H}^k(\Omega)}\big)
\big\|\partial_\eta({\bm v}^n-\bar{\bm v})(\tau)\big\|_{\mathcal{H}^k(\Omega)}.
\end{aligned}\end{align}

\textbf{Step 2.} In this step, we will estimate the terms on the right hand side of \eqref{partial-iden}. Firstly, we give the following claim to be proved later.
\begin{claim}\label{claim2}
\begin{align*}
&\Big|\left\langle\partial^\alpha({\bm v}^n-\bar{\bm v})^{\it T}\bm{S}({\bm v}^{n-1}),
[\partial^\alpha,\bm{F}({\bm v}^{n-1},\partial_\eta {\bm v}^{n-1})]\partial_\eta {\bm v}^n\right\rangle(\tau)\nonumber\\
&\quad+\int_\bT\partial^\alpha({\bm v}^n-\bar{\bm v})^{\it T}
[\bm{SF}]({\bm v}^{n-1},\partial^\alpha {\bm v}^{n-1})\partial_\eta {\bm v}^n\big|_{\eta=0}d\xi\Big|\nonumber\\
&\lesssim \big(1+\|({\bm v}^{n-1}-\bar{\bm v})(\tau)\|_{\mathcal{H}^k(\Omega)}^5\big)\big(1+\|({\bm v}^n-\bar{\bm v})(\tau)\|_{\mathcal{H}^k(\Omega)}\big)\nonumber\\
&\quad\cdot\big(\|({\bm v}^n-\bar{\bm v})(\tau)\|_{\mathcal{H}^k(\Omega)}
+\|\partial_\eta({\bm v}^n-\bar{\bm v})(\tau)\|_{\mathcal{H}^k(\Omega)}\big).
\end{align*}
\end{claim}

Next, we turn to estimate the remaining terms on the right hand side of (\ref{partial-iden}) involving commutators.
It holds that by using \eqref{moser2} and the definition \eqref{AB} of $\bm A$,
\begin{align}\label{cA-est}
\begin{aligned}
&\left|\left\langle\partial^\alpha({\bm v}^n-\bar{\bm v})^{\it T}\bm{S}({\bm v}^{n-1}),
[\partial^\alpha,\bm{A}({\bm v}^{n-1})]\partial_\xi({\bm v}^n-\bar{\bm v})\right\rangle(\tau)\right|\\
	&\leq\|\bm{S}({\bm v}^{n-1})(\tau,\cdot)\|_{L^\infty(\Omega)}\|\partial^\alpha({\bm v}^n-\bar{\bm v})(\tau,\cdot)\|_{L^2(\Omega)}
\left\|[\partial^\alpha,\bm{A}({\bm v}^{n-1})]\partial_\xi({\bm v}^n-\bar{\bm v})(\tau,\cdot)\right\|_{L^2(\Omega)}\\
	&\lesssim \big(1+\|({\bm v}^{n-1}-\bar{\bm v})(\tau)\|_{H^2(\Omega)}^3\big)\big\|({\bm v}^n-\bar{\bm v})(\tau)\big\|_{\mathcal{H}^k(\Omega)}\\
	&\qquad\cdot
	\big(1+\|({\bm v}^{n-1}-\bar{\bm v})(\tau)\|_{\mathcal{H}^k(\Omega)}^3\big)\big\|\partial_\xi({\bm v}^n-\bar{\bm v})(\tau)\big\|_{\mathcal{H}^{k-1}(\Omega)}\\
	&\lesssim \big(1+\|({\bm v}^{n-1}-\bar{\bm v})(\tau)\|_{\mathcal{H}^k(\Omega)}^6\big)\big\|({\bm v}^n-\bar{\bm v})(\tau)\big\|_{\mathcal{H}^k(\Omega)}^2.
\end{aligned}\end{align}
Similarly, one has 
\begin{align}\label{cB-est}
\begin{aligned}
&\left|\left\langle\partial^\alpha({\bm v}^n-\bar{\bm v})^{\it T}\bm{S}({\bm v}^{n-1}),
[\partial^\alpha,\bm{B}({\bm v}^{n-1})]\partial_{\eta}^2({\bm v}^n-\bar{\bm v})\right\rangle(\tau)\right|\\
&\leq\|\bm{S}({\bm v}^{n-1})(\tau,\cdot)\|_{L^\infty(\Omega)}\|\partial^\alpha({\bm v}^n-\bar{\bm v})(\tau,\cdot)\|_{L^2(\Omega)}
\left\|[\partial^\alpha,\bm{B}({\bm v}^{n-1})]\partial_\eta^2({\bm v}^n-\bar{\bm v})(\tau,\cdot)\right\|_{L^2(\Omega)}\\
&\lesssim \big(1+\|({\bm v}^{n-1}-\bar{\bm v})(\tau)\|_{H^2(\Omega)}^3\big)\big\|({\bm v}^n-\bar{\bm v})(\tau)\big\|_{\mathcal{H}^k(\Omega)}\\
&\qquad\cdot\big(1+\|({\bm v}^{n-1}-\bar{\bm v})(\tau)\|_{\mathcal{H}^k(\Omega)}^5\big)\big\|\partial_\eta^2({\bm v}^n-\bar{\bm v})(\tau)\big\|_{\mathcal{H}^{k-1}(\Omega)}\\
&\lesssim \big(1+\|({\bm v}^{n-1}-\bar{\bm v})(\tau)\|_{\mathcal{H}^k(\Omega)}^8\big)\big\|({\bm v}^n-\bar{\bm v})(\tau)\big\|_{\mathcal{H}^k(\Omega)}
\big\|\partial_\eta({\bm v}^n-\bar{\bm v})(\tau)\big\|_{\mathcal{H}^k(\Omega)}.
\end{aligned}\end{align}

Further, recalling $\bm{r}$ given in \eqref{def-G}, it is easy to verify that
\begin{align}\label{G-est}
\begin{aligned}
&\left|\left\langle\partial^\alpha({\bm v}^n-\bar{\bm v})^{\it T}\bm{S}({\bm v}^{n-1}), \partial^\alpha\bm{r}\right\rangle(\tau)\right|\\
&\leq\|\bm{S}({\bm v}^{n-1})(\tau,\cdot)\|_{L^\infty(\Omega)}\|\partial^\alpha({\bm v}^n-\bar{\bm v})(\tau,\cdot)\|_{L^2(\Omega)}
\big\|\partial^\alpha\bm{r}(\tau,\cdot)\big\|_{L^2(\Omega)}\\
&\lesssim \big(1+\|({\bm v}^{n-1}-\bar{\bm v})(\tau)\|_{\mathcal{H}^k(\Omega)}^8\big)\big\|({\bm v}^n-\bar{\bm v})(\tau)\big\|_{\mathcal{H}^k(\Omega)}.
\end{aligned}\end{align}

\textbf{Step 3.} Now, we can obtain the desired estimate. 
Indeed, by plugging the estimates (\ref{S-est})-(\ref{P-est}), (\ref{B-est}), Claim \ref{claim2}, \eqref{cA-est}-\eqref{G-est} into \eqref{partial-iden},
it follows that for $|\alpha|\leq k,\ k\geq 4$,
\begin{align*}
&\frac{1}{2}\frac{d}{d\tau}\sum_{|\alpha|\leq k}\left\langle\partial^\alpha({\bm v}^n-\bar{\bm v})^{\it T}\bm{S}({\bm v}^{n-1}),
\partial^\alpha({\bm v}^n-\bar{\bm v})\right\rangle(\tau)
+c_\delta\|\partial_\eta({\bm v}^n-\bar{\bm v})(\tau)\|^2_{\mathcal{H}^k(\Omega)}\nonumber\\
&\lesssim \big(1+\|({\bm v}^{n-1}-\bar{\bm v})(\tau)\|_{\mathcal{H}^k(\Omega)}^{9}\big)\big(1+\|({\bm v}^n-\bar{\bm v})(\tau)\|_{\mathcal{H}^k(\Omega)}\big)\nonumber\\
&\qquad\cdot
\big(\|({\bm v}^n-\bar{\bm v})(\tau)\|_{\mathcal{H}^k(\Omega)}+\|\partial_\eta({\bm v}^n-\bar{\bm v})(\tau)\|_{\mathcal{H}^k(\Omega)}\big)\nonumber\\
&\leq \frac{c_\delta}{2}\|\partial_\eta({\bm v}^n-\bar{\bm v})(\tau)\|_{\mathcal{H}^k(\Omega)}^2
+C\big(1+\|({\bm v}^{n-1}-\bar{\bm v})(\tau)\|_{\mathcal{H}^k(\Omega)}^{18}\big)\big(1+\|({\bm v}^n-\bar{\bm v})(\tau)\|^2_{\mathcal{H}^k(\Omega)}\big),
\end{align*}
which implies that
\begin{align}\label{energy}
\begin{aligned}
&\frac{d}{d\tau}\sum_{|\alpha|\leq k}\left\langle\partial^\alpha({\bm v}^n-\bar{\bm v})^{\it T}\bm{S}({\bm v}^{n-1}),
\partial^\alpha({\bm v}^n-\bar{\bm v})\right\rangle(\tau)
+\|\partial_\eta({\bm v}^n-\bar{\bm v})(\tau)\|^2_{\mathcal{H}^k(\Omega)}\\
&\leq C\big(1+\|({\bm v}^{n-1}-\bar{\bm v})(\tau)\|_{\mathcal{H}^k(\Omega)}^{18}\big)
\big(1+\|({\bm v}^n-\bar{\bm v})(\tau)\|^2_{\mathcal{H}^k(\Omega)}\big).
\end{aligned}\end{align}
Since $\bm{S}(\cdot)$ is positive-definite, together with the assumption $\eqref{induc-hypo}_3$,
one has
\begin{align*}
&\left\langle\partial^\alpha({\bm v}^n-\bar{\bm v})^{\it T}\bm{S}({\bm v}^{n-1}),\partial^\alpha({\bm v}^n-\bar{\bm v})\right\rangle(\tau)
\geq c_\delta\|\partial^\alpha({\bm v}^n-\bar{\bm v})(\tau,\cdot)\|^2_{L^2(\Omega)},
\end{align*}
and then,
\begin{align}\label{S_positive}
&\sum_{|\alpha|\leq k}\left\langle\partial^\alpha({\bm v}^n-\bar{\bm v})^{\it T}\bm{S}({\bm v}^{n-1}),\partial^\alpha({\bm v}^n-\bar{\bm v})\right\rangle(\tau)
\geq c_\delta\|({\bm v}^n-\bar{\bm v})(\tau)\|^2_{\mathcal{H}^k(\Omega)}.
\end{align}
Using (\ref{S_positive}), we apply the Gronwall inequality to (\ref{energy}) to obtain that for $t\in[0,T_1],$
\begin{align}\label{energy-1}
\begin{aligned}
&\sum_{|\alpha|\leq k}\left\langle\partial^\alpha({\bm v}^n-\bar{\bm v})^{\it T}\bm{S}({\bm v}^{n-1}),\partial^\alpha({\bm v}^n-\bar{\bm v})\right\rangle(t)
+\int_{0}^{t}\big\|\partial_\eta({\bm v}^n-\bar{\bm v})(\tau)\big\|^2_{\mathcal{H}^k(\Omega)}d\tau\\
\leq~& C\Big(1+\sum_{|\alpha|\leq k}\left\langle\partial^\alpha({\bm v}^n-\bar{\bm v})^{\it T}\bm{S}({\bm v}^{n-1}),
\partial^\alpha({\bm v}^n-\bar{\bm v})\right\rangle(0)\Big)\\
&\qquad\cdot\exp\Big\{C\int_{0}^{t}\big\|({\bm v}^{n-1}-\bar{\bm v})(\tau)\big\|_{\mathcal{H}^k(\Omega)}^{18}d\tau\Big\}.
\end{aligned}\end{align}
Meanwhile, from $\eqref{induc-hypo}_1$ and \eqref{n-initial}:
\begin{align*}
  \partial^j_\tau {\bm v}^{n-1}|_{\tau=0}=\partial^j_\tau \bm v^{n}|_{\tau=0}={\bm v}_0^j(\xi,\eta),\quad 0\leq j\leq k,
\end{align*}
 along with \eqref{est_V}, it yields that
\begin{align}\label{est-initial}
\sum_{|\alpha|\leq k}\left\langle\partial^\alpha({\bm v}^n-\bar{\bm v})^{\it T}\bm{S}({\bm v}^{n-1}),
\partial^\alpha({\bm v}^n-\bar{\bm v})\right\rangle(0)\lesssim M_1^3.
\end{align}
Substituting \eqref{est-initial} into \eqref{energy-1} and using \eqref{S_positive} again,
we obtain that there exists a constant $C_0>0$ such that
\begin{align*}
\begin{aligned}
&\|{\bm v}^n-\bar{\bm v}\|^2_{\mathcal{H}^k(\Omega_t)}
+\int_{0}^{t}\big\|\partial_\eta({\bm v}^n-\bar{\bm v})(\tau)\big\|^2_{\mathcal{H}^k(\Omega)}d\tau\\
&\leq C_0M_1^3~ \exp\Big\{C_0\int_{0}^{t}\big\|({\bm v}^{n-1}-\bar{\bm v})(\tau)\big\|_{\mathcal{H}^k(\Omega)}^{18}d\tau\Big\}.
\end{aligned}\end{align*}

\end{proof}

To complete the proof of Proposition~\ref{energy-est}, it remains to show Claim \ref{claim1} and Claim \ref{claim2}. 
\begin{proof}[Proof of Claim \ref{claim1}.]
Denoted by
\begin{align}\label{def_I}
I:=\int_\bT\partial^\alpha({\bm v}^n-\bar{\bm v})^{\it T}[\bm{SB}]({\bm v}^{n-1})\partial_\eta\partial^\alpha({\bm v}^n-\bar{\bm v})\big|_{\eta=0}d\xi.
\end{align}
First of all, for the case of $\partial^\alpha=\partial^{\alpha_0}_\tau\partial^{\alpha_1}_\xi$, 
 the boundary conditions on $\{\eta=0\}$ in 
 $(\ref{lead-iden})$ gives
 \begin{align}\label{bd-alpha}
 	\partial^\alpha({\bm v}^n-\bar{\bm v})\big|_{\eta=0}=\big(0,0,\partial^\alpha w^n\big)(\tau,\xi,0),
 	\end{align}
 	and
 	\begin{align*}
 	\partial_\eta\partial^\alpha({\bm v}^n-\bar{\bm v})\big|_{\eta=0}=\big(\partial_\eta\partial^\alpha v^n,~\partial_\eta\partial^\alpha\vartheta^n,~0\big)(\tau,\xi,0).
 \end{align*}
Then combining with the expression \eqref{SB} of $[\bm{SB}]$, it is easy to calculate that
\begin{equation}\label{I=0}
I=0.
\end{equation}
Next, for the term $\int_\bT\partial^\alpha({\bm v}^n-\bar{\bm v})^{\it T}[\bm{SF}]({\bm v}^{n-1}, \partial^\alpha {\bm v}^{n-1})\partial_\eta {\bm v}^n|_{\eta=0}d\xi,$
from the definition \eqref{SF} of $[\bm{SF}]$ and the boundary conditions $\eqref{induc-hypo}_2$ we know that at $\{\eta=0\}$,
\begin{align*}
	&[\bm{SF}](\bm v^{n-1},\partial^\alpha \bm v^{n-1})\\
	&=\left(
\begin{array}{ccc}
(\frac{\nu(P-q^{n-1})}{R}-\mu\theta^{n-1})\theta^{n-1}\partial^\alpha q^{n-1} & 0 & 0\\
-2\mu\theta^{n-1} q^{n-1}\partial^\alpha u_1^{n-1} & -\kappa\theta^{n-1}\partial^\alpha q^{n-1} & \frac{\nu(P-q^{n-1})}{a}\partial^\alpha \theta^{n-1}\\
0 & 0 & \nu\theta^{n-1}\big(\partial^\alpha \theta^{n-1}+\frac{\theta^{n-1}(P+q^{n-1})}{2q^{n-1}(P-q^{n-1})}\partial^\alpha q^{n-1}\big)
\end{array}\right)\\
&=\left(
\begin{array}{ccc}
(\frac{\nu(P-q^{n-1})}{R}-\mu\theta^{*})\theta^{*}\partial^\alpha q^{n-1} & 0 & 0\\
0 & -\kappa\theta^{*}\partial^\alpha q^{n-1} & \frac{\nu(P-q^{n-1})}{a}\partial^\alpha \theta^{*}\\
0 & 0 & \nu\theta^{*}\big(\partial^\alpha \theta^{*}+\frac{\theta^{*}(P+q^{n-1})}{2q^{n-1}(P-q^{n-1})}\partial^\alpha q^{n-1}\big)
\end{array}\right).
\end{align*}
Then, by \eqref{bd-alpha} and 
\begin{align*}
 	\partial_\eta{\bm v}^n\big|_{\eta=0}=\big(\partial_\eta u_1^n,~\partial_\eta \theta^n,~0\big)(\tau,\xi,0),
 \end{align*}
  direct calculation yields that
\begin{equation}\label{I_0}
 \int_\bT\partial^\alpha({\bm v}^n-\bar{\bm v})^{\it T}[\bm{SF}]({\bm v}^{n-1}, \partial^\alpha {\bm v}^{n-1})\partial_\eta {\bm v}^n\big|_{\eta=0}d\xi=0,
 \end{equation}
and that is why we choose $\bm{F}$ in the form \eqref{F}. 
In this case, Claim \ref{claim1} follows automatically from \eqref{I=0} and \eqref{I_0}. 

Now, we only need to investigate the case of  $\partial^\alpha=\partial_\eta\partial^\beta,\ |\beta|\leq k-1$. By using the identity (\ref{partial-iden}) to replace
$$\bm{B}({\bm v}^{n-1})\partial_{\eta}\partial^\alpha ({\bm v}^n-\bar{\bm v})=\bm{B}({\bm v}^{n-1})\partial_{\eta}^2\partial^\beta ({\bm v}^n-\bar{\bm v})$$ 
in \eqref{def_I} to obtain that
\begin{align}\label{def-I}
\begin{aligned}
	I&=\int_\bT\partial^\alpha({\bm v}^n-\bar{\bm v})^{\it T}
	\Big\{\bm{S}({\bm v}^{n-1})\partial_\tau\partial^\beta({\bm v}^n-\bar{\bm v})\\
	&\qquad\quad+[\bm{SA}]({\bm v}^{n-1})\partial_\xi\partial^\beta({\bm v}^n-\bar{\bm v})
+[\bm{SF}]({\bm v}^{n-1},\partial_\eta {\bm v}^{n-1})\partial_\eta\partial^\beta {\bm v}^n\Big\}\Big|_{\eta=0}d\xi\\
	&\quad-\int_\bT\partial^\alpha({\bm v}^n-\bar{\bm v})^{\it T}\bm{S}({\bm v}^{n-1})
[\partial^\beta,\bm{B}({\bm v}^{n-1})]\partial_{\eta}^2{\bm v}^n\big|_{\eta=0}d\xi\\
	&\quad+\int_\bT\partial^\alpha({\bm v}^n-\bar{\bm v})^{\it T}\bm{S}({\bm v}^{n-1})\\
&\qquad\cdot\Big\{\partial^\beta\big(\bm{G}({\bm v}^{n-1})({\bm v}^n-\bar{\bm v})\big)
+[\partial^\beta,\bm{A}({\bm v}^{n-1})]\partial_\xi({\bm v}^n-\bar{\bm v})\Big\}\Big|_{\eta=0}d\xi\\
	&\quad+\int_\bT\partial^\alpha({\bm v}^n-\bar{\bm v})^{\it T}\bm{S}({\bm v}^{n-1})
	[\partial^\beta,\bm{F}({\bm v}^{n-1},\partial_\eta {\bm v}^{n-1})]\partial_\eta {\bm v}^n\big|_{\eta=0}d\xi\\
	&\quad+\int_\bT\partial^\alpha({\bm v}^n-\bar{\bm v})^{\it T}\bm{S}({\bm v}^{n-1})
\partial^\beta\bm{r}\big|_{\eta=0}d\xi=:\sum_{i=1}^5I_i.
\end{aligned}\end{align}
Firstly, under the induction hypothesis $(u_1^{n-1},\partial_\eta q^{n-1})\big|_{\eta=0}=\mathbf{0}$ given in \eqref{induc-hypo}, and using $\phi(\eta)\equiv0$ for $\eta\leq1$, it follows that
\begin{align}\label{I1}
\begin{aligned}
	|I_1|	&\lesssim \big(1+\big\|({\bm v}^{n-1},\partial_\eta {\bm v}^{n-1})\big|_{\eta=0}\big\|_{L^\infty(\bT_\xi)}^3\big)
	\big\|\partial\partial^\beta ({\bm v}^n-\bar{\bm v})\big|_{\eta=0}\big\|_{L^2(\bT_\xi)}^2\\
	&\lesssim \big(1+\|({\bm v}^{n-1}-\bar{\bm v})(\tau)\|_{H^3(\Omega)}^3\big)\big\|({\bm v}^n-\bar{\bm v})(\tau)\big\|_{\mathcal{H}^k(\Omega)}
	\big\|\partial_\eta({\bm v}^n-\bar{\bm v})(\tau)\big\|_{\mathcal{H}^k(\Omega)},
\end{aligned}\end{align}
where $\displaystyle \partial:=\sum_{|\alpha|=1}\partial^\alpha$ and we have used \eqref{trace}.

Secondly, it is easy to get that for $I_2$,
\begin{align*}
	|I_{2}|&\leq\big\|S({\bm v}^{n-1})\big|_{\eta=0}\big\|_{L^\infty(\mathbb{T}_\xi)}
\big\|\partial^\alpha (\bm v^{n}-\bar{\bm v})\big|_{\eta=0}\big\|_{L^2(\mathbb{T}_\xi)}
	\big\|[\partial^\beta,\bm{B}({\bm v}^{n-1})] \partial_{\eta}^2 (\bm v^{n}-\bar{\bm v})\big|_{\eta=0}
\big\|_{L^2(\mathbb{T}_\xi)}.
\end{align*}
With the help of \eqref{trace}, one can obtain that
\begin{align}\label{moser-type}
\begin{aligned}
	&\big\|[\partial^\beta,\bm{B}({\bm v}^{n-1})]\partial_{\eta}^2({\bm v}^n-\bar{\bm v})\big|_{\eta=0}\big\|_{L^2(\bT_\xi)}\\
&\leq\sum_{1\leq|\beta'|\leq|\beta|-1}\Big\{C_{\beta}^{\beta'}\big\|\partial^{\beta'}\big(\bm{B}({\bm v}^{n-1})\big)\big|_{\eta=0}\big\|_{L^\infty(\bT_\xi)}
\cdot\big\|\partial_{\eta}^2\partial^{\beta-\beta'}({\bm v}^n-\bar{\bm v})\big|_{\eta=0}\big\|_{L^2(\bT_\xi)}\Big\}\\
&\qquad+\big\|\partial^\beta\big(\bm{B}({\bm v}^{n-1})\big)\big|_{\eta=0}\big\|_{L^2(\bT_\xi)}\cdot
\big\|\partial_{\eta}^2({\bm v}^n-\bar{\bm v}\big)\big|_{\eta=0}\|_{L^\infty(\bT_\xi)})\\
&\lesssim \big(1+\|({\bm v}^{n-1}-\bar{\bm v})(\tau)\|_{\mathcal{H}^k(\Omega)}^5\big)\cdot\Big(\big\|\partial_{\eta}^2({\bm v}^n-\bar{\bm v})(\tau)\big\|_{H^2(\Omega)}\\
&\qquad\qquad+\big\|\partial_{\eta}^2({\bm v}^n-\bar{\bm v})(\tau)\big\|_{\mathcal{H}^{k-2}(\Omega)}^\frac{1}{2}
\big\|\partial_{\eta}^3({\bm v}^n-\bar{\bm v})(\tau)\big\|_{\mathcal{H}^{k-2}(\Omega)}^\frac{1}{2}\Big)\\
&\lesssim \big(1+\|({\bm v}^{n-1}-\bar{\bm v})(\tau)\|^5_{\mathcal{H}^k(\Omega)}\big)\cdot\Big(\big\|({\bm v}^n-\bar{\bm v})(\tau)\big\|_{\mathcal{H}^k(\Omega)}\\
&\qquad\qquad +\big\|({\bm v}^n-\bar{\bm v})(\tau)\big\|_{\mathcal{H}^k(\Omega)}^{\frac{1}{2}}
\big\|\partial_\eta({\bm v}^n-\bar{\bm v})(\tau)\big\|_{\mathcal{H}^k(\Omega)}^{\frac{1}{2}}\Big).
\end{aligned}\end{align}
Thus, it follows that
\begin{align}\label{I2}
\begin{aligned}
		|I_{2}|&\lesssim \big(1+\|({\bm v}^{n-1}-\bar{\bm v})(\tau)\|_{H^2(\Omega)}^8\big)\big\|\partial^\alpha({\bm v}^n-\bar{\bm v})(\tau,\cdot)\big\|_{L^2(\Omega)}^\frac{1}{2}
	\big\|\partial_{\eta}\partial^\alpha({\bm v}^n-\bar{\bm v})(\tau,\cdot)\big\|_{L^2(\Omega)}^\frac{1}{2}\\
	&\quad\cdot 
	\Big(\big\|({\bm v}^n-\bar{\bm v})(\tau)\|_{\mathcal{H}^k(\Omega)}+\big\|({\bm v}^n-\bar{\bm v})(\tau)\big\|_{\mathcal{H}^k(\Omega)}^\frac{1}{2}
	\big\|\partial_\eta({\bm v}^n-\bar{\bm v})(\tau)\big\|_{\mathcal{H}^k(\Omega)}^\frac{1}{2}\Big)\\
	&\lesssim \big(1+\|({\bm v}^{n-1}-\bar{\bm v})(\tau)\|^8_{\mathcal{H}^k(\Omega)}\big)\big\|({\bm v}^n-\bar{\bm v})(\tau)\big\|_{\mathcal{H}^k(\Omega)}\\
	&\quad\cdot\big(\big\|({\bm v}^n-\bar{\bm v})(\tau)\big\|_{\mathcal{H}^k(\Omega)}+
	\big\|\partial_\eta({\bm v}^n-\bar{\bm v})(\tau)\big\|_{\mathcal{H}^k(\Omega)}\big).
\end{aligned}
\end{align}
Similarly, it holds that
\begin{align}\label{I3}
|I_{3}|&\lesssim \big(1+\|({\bm v}^{n-1}-\bar{\bm v})(\tau)\|_{\mathcal{H}^k(\Omega)}^6\big)
	\big\|({\bm v}^n-\bar{\bm v})(\tau)\big\|_{\mathcal{H}^k(\Omega)}\big\|\partial_\eta({\bm v}^n-\bar{\bm v})(\tau)\big\|_{\mathcal{H}^k(\Omega)}.
\end{align}

Thirdly, to estimate $I_4$, note that $\partial^\alpha=\partial_\eta\partial^\beta,$
\begin{align}\label{div-F}
	\bm{S}({\bm v}^{n-1})[\partial^\beta,\bm{F}({\bm v}^{n-1},\partial_\eta {\bm v}^{n-1})]\partial_\eta {\bm v}^n
	-[\bm{SF}]({\bm v}^{n-1}, \partial^\alpha {\bm v}^{n-1})\partial_\eta {\bm v}^n=:\mathcal{R},
\end{align}
where the remainder $\mathcal{R}$ is a linear combination of
\begin{align*}
	[\bm{SF}]({\bm v}^{n-1},\partial_\eta \partial^{\beta'} {\bm v}^{n-1})
	\partial_\eta\partial^{\beta-\beta'}{\bm v}^n,\quad \beta'\leq\beta,~1\leq|\beta'|\leq|\beta|-1,
\end{align*}
and
\begin{align*}
&\bm{S}({\bm v}^{n-1})
\tilde{\bm{F}}(\partial^{\beta_1}{\bm v}^{n-1},\partial_\eta\partial^{\beta_2} {\bm v}^{n-1})\partial_\eta\partial^{\beta_3}{\bm v}^n,\\
&\quad \beta_1+\beta_2+\beta_3<\beta,~or~\beta_1+\beta_2+\beta_3=\beta~\mbox{with}~\beta_1\geq1,
\end{align*}
for some $\tilde{\bm{F}}(\cdot,\cdot)$ defined similar to $\bm{F}(\cdot,\cdot)$. On the one hand,
for $\beta'\leq\beta, 1\leq |\beta'|\leq|\beta|-1$, similar as \eqref{moser-type} one has
\begin{align*}
	&\Big|\int_\bT\partial^\alpha({\bm v}^n-\bar{\bm v})^{\it T}[\bm{SF}]({\bm v}^{n-1},\partial_\eta\partial^{\beta'} {\bm v}^{n-1})\partial_\eta\partial^{\beta-\beta'} {\bm v}^n|_{\eta=0}d\xi\Big|\nonumber\\
	&\lesssim \big\|\partial^\alpha({\bm v}^n-\bar{\bm v})(\tau,\cdot)\big\|_{L^2(\Omega)}^\frac{1}{2}
	\big\|\partial_{\eta}\partial^\alpha({\bm v}^n-\bar{\bm v})(\tau,\cdot)\big\|_{L^2(\Omega)}^\frac{1}{2}\\
&\quad\cdot	\big(1+\|({\bm v}^{n-1}-\bar{\bm v})(\tau,\cdot)\|_{L^\infty(\Omega)}^3\big)\big(1+\|({\bm v}^{n-1}-\bar{\bm v})(\tau)\|_{\mathcal{H}^k(\Omega)}\big)\\
&\qquad\cdot \big\|\partial_\eta({\bm v}^n-\bar{\bm v})(\tau,\cdot)\big\|_{\mathcal{H}^{k-2}(\Omega)}^\frac{1}{2}
	\big\|\partial_{\eta}^2({\bm v}^n-\bar{\bm v})(\tau,\cdot)\big\|_{\mathcal{H}^{k-2}(\Omega)}^\frac{1}{2}\\
	&\lesssim \big(1+\|({\bm v}^{n-1}-\bar{\bm v})(\tau)\|_{\mathcal{H}^k(\Omega)}^4\big)\big\|({\bm v}^n-\bar{\bm v})(\tau)\big\|_{\mathcal{H}^k(\Omega)}\big\|\partial_\eta({\bm v}^n-\bar{\bm v})(\tau)\big\|_{\mathcal{H}^k(\Omega)}.
\end{align*}
On the other hand, it also holds
\begin{align*}
	&\Big|\int_\bT\partial^\alpha({\bm v}^n-\bar{\bm v})^{\it T}{[\bm{S}\tilde{\bm F}]}(\partial^{\beta_1}{\bm v}^{n-1},\partial_\eta\partial^{\beta_2} {\bm v}^{n-1})\partial_\eta\partial^{\beta_3} {\bm v}^n\big|_{\eta=0}d\xi\Big|\nonumber\\
	&\lesssim \big\|\partial^\alpha({\bm v}^n-\bar{\bm v})(\tau,\cdot)\big\|_{L^2(\Omega)}^\frac{1}{2}
	\big\|\partial_{\eta}\partial^\alpha({\bm v}^n-\bar{\bm v})(\tau,\cdot)\big\|_{L^2(\Omega)}^\frac{1}{2}\\
	&\quad\cdot\big(1+\|({\bm v}^{n-1}-\bar{\bm v})(\tau,\cdot)\|_{L^\infty(\Omega)}^3\big)\big(1+\|({\bm v}^{n-1}-\bar{\bm v})(\tau)\|_{\mathcal{H}^k(\Omega)}^6\big)\\
	&\qquad\cdot \big\|\partial_\eta({\bm v}^n-\bar{\bm v})(\tau,\cdot)\big\|_{\mathcal{H}^{k-2}(\Omega)}^\frac{1}{2}
	\big\|\partial_{\eta}^2({\bm v}^n-\bar{\bm v})(\tau,\cdot)\big\|_{\mathcal{H}^{k-2}(\Omega)}^\frac{1}{2}\\
	&\lesssim \big(1+\|({\bm v}^{n-1}-\bar{\bm v})(\tau)\|_{\mathcal{H}^k(\Omega)}^{9}\big)\big\|({\bm v}^n-\bar{\bm v})(\tau)\big\|_{\mathcal{H}^k(\Omega)}\big\|\partial_\eta({\bm v}^n-\bar{\bm v})(\tau)\big\|_{\mathcal{H}^k(\Omega)}.
\end{align*}
Then gethering \eqref{div-F} with the above two inequalities yields
\begin{align}\label{I4}
\begin{aligned}
	&\Big|I_4-\int_\bT\partial^\alpha({\bm v}^n-\bar{\bm v})^{\it T}
	[\bm{SF}]({\bm v}^{n-1},\partial^\alpha {\bm v}^{n-1})\partial_\eta {\bm v}^n\big|_{\eta=0}d\xi\Big|\\
	&\lesssim \big(1+\|({\bm v}^{n-1}-\bar{\bm v})(\tau)\|_{\mathcal{H}^k(\Omega)}^{9}\big)
\big\|({\bm v}^n-\bar{\bm v})(\tau)\big\|_{\mathcal{H}^k(\Omega)}\big\|\partial_\eta({\bm v}^n-\bar{\bm v})(\tau)\big\|_{\mathcal{H}^k(\Omega)}.
\end{aligned}\end{align}

Finally, it is straightforward to show that
\begin{align}\label{I5}
\begin{aligned}
	|I_5|&\lesssim \big\|\partial^\alpha({\bm v}^n-\bar{\bm v})(\tau,\cdot)\big\|_{L^2(\Omega)}^\frac{1}{2}
	\big\|\partial_{\eta}\partial^\alpha({\bm v}^n-\bar{\bm v})(\tau,\cdot)\big\|_{L^2(\Omega)}^\frac{1}{2}\\&\quad\cdot\big(1+\|({\bm v}^{n-1}-\bar{\bm v})(\tau,\cdot)\|_{L^\infty(\Omega)}^3\big)\big(1+\|({\bm v}^{n-1}-\bar{\bm v})(\tau)\|_{\mathcal{H}^k(\Omega)}^5\big)\\
	&\lesssim \big(1+\|({\bm v}^{n-1}-\bar{\bm v})(\tau)\|_{\mathcal{H}^k(\Omega)}^8\big)\big(1+\|({\bm v}^n-\bar{\bm v})(\tau)\|_{\mathcal{H}^k(\Omega)}\big)
	\big\|\partial_\eta({\bm v}^n-\bar{\bm v})(\tau)\big\|_{\mathcal{H}^k(\Omega)}.
\end{aligned}\end{align}
Therefore, substituting the above estimates \eqref{I1}, \eqref{I2}, \eqref{I3}, \eqref{I4} and \eqref{I5} for $I_i,\ 1\leq i\leq 5$ respectively, into \eqref{def-I} leads to that for $\partial^\alpha=\partial_\eta\partial^\beta$,
\begin{align*}
\begin{aligned}
&\Big|I-\int_\bT\partial^\alpha({\bm v}^n-\bar{\bm v})^{\it T}[\bm{SF}]({\bm v}^{n-1},\partial^\alpha {\bm v}^{n-1})\partial_\eta {\bm v}^n|_{\eta=0}d\xi\Big|\\
	&\lesssim \big(1+\|({\bm v}^{n-1}-\bar{\bm v})(\tau)\|_{\mathcal{H}^k(\Omega)}^{9}\big)\big(1+\|({\bm v}^n-\bar{\bm v})(\tau)\|_{\mathcal{H}^k(\Omega)}\big)
	\big\|\partial_\eta({\bm v}^n-\bar{\bm v})(\tau)\big\|_{\mathcal{H}^k(\Omega)},
\end{aligned}\end{align*}
which eventually gives the proof of Claim \ref{claim1}.

\end{proof}

Next, we prove Claim \ref{claim2} as follows.
\begin{proof}[Proof of Claim \ref{claim2}]
	Similar to \eqref{div-F}, one has
	\begin{align}\label{div-F1}
		[\partial^\alpha,\bm{F}({\bm v}^{n-1},\partial_\eta {\bm v}^{n-1})]\partial_\eta {\bm v}^n
		=\bm{F}({\bm v}^{n-1},\partial_\eta\partial^\alpha {\bm v}^{n-1})\partial_\eta {\bm v}^n+\mathcal{R}_1,
	\end{align}
	where the remainder $\mathcal{R}_1$ is a linear combination of
	\begin{align*}
		\bm{F}({\bm v}^{n-1},\partial_\eta \partial^{\alpha'} {\bm v}^{n-1})
		\partial_\eta\partial^{\alpha-\alpha'}{\bm v}^n,\quad \alpha'\leq\alpha,~1\leq|\alpha'|\leq|\alpha|-1,
	\end{align*}
	and
\begin{align*}
&\tilde{\bm{F}}(\partial^{\alpha_1}{\bm v}^{n-1},\partial_\eta\partial^{\alpha_2}{\bm v}^{n-1})\partial_\eta\partial^{\alpha_3}{\bm v}^n,\\
&\quad\alpha_1+\alpha_2+\alpha_3<\alpha,~or~\alpha_1+\alpha_2+\alpha_3=\alpha~\mbox{with}~\alpha_1\geq1.
\end{align*}
To estimate $\mathcal{R}_1$, on the one hand, by \eqref{moser0} we can claim that for $\alpha'\leq\alpha$ and $1\leq|\alpha'|\leq|\alpha|-1,$
\begin{align*}
&\Big|\left\langle\partial^\alpha({\bm v}^n-\bar{\bm v})^{\it T}\bm{S}({\bm v}^{n-1}),
\bm{F}({\bm v}^{n-1},\partial_\eta\partial^{\alpha'} {\bm v}^{n-1})\partial_\eta\partial^{\alpha-\alpha'}
{\bm v}^n\right\rangle(\tau)\Big|\nonumber\\
&\lesssim \big\|\partial^\alpha({\bm v}^n-\bar{\bm v})(\tau,\cdot)\big\|_{L^2(\Omega)}
\big(1+\|({\bm v}^{n-1}-\bar{\bm v})(\tau)\|_{\mathcal{H}^k(\Omega)}^6\big)\big(1+\|\partial_\eta({\bm v}^n-\bar{\bm v})(\tau)\|_{\mathcal{H}^{k-1}(\Omega)}\big).
\end{align*}
On the other hand, for $\alpha_1+\alpha_2+\alpha_3<\alpha$ or $\alpha_1+\alpha_2+\alpha_3=\alpha~\mbox{with}~\alpha_1\geq1$,
\begin{align*}
&\Big|\left\langle\partial^\alpha({\bm v}^n-\bar{\bm v})^{\it T}\bm{S}({\bm v}^{n-1}),
{\tilde{\bm F}}(\partial^{\alpha_1}{\bm v}^{n-1},\partial_\eta\partial^{\alpha_2} {\bm v}^{n-1})\partial_\eta\partial^{\alpha_3} {\bm v}^n\right\rangle(\tau)\Big|\nonumber\\
&\lesssim \big\|\partial^\alpha({\bm v}^n-\bar{\bm v})(\tau,\cdot)\|_{L^2(\Omega)}
\big(1+\|({\bm v}^{n-1}-\bar{\bm v})(\tau)\|_{\mathcal{H}^k(\Omega)}^9\big)\big(1+\|\partial_\eta({\bm v}^n-\bar{\bm v})(\tau)\|_{\mathcal{H}^{k-1}(\Omega)}\big).
\end{align*}
Thus, combining with the above two inequalities yields
\begin{align}\label{R-est}
\begin{aligned}
&\big|\left\langle\partial^\alpha({\bm v}^n-\bar{\bm v})^{\it T}\bm{S}({\bm v}^{n-1}),\mathcal{R}_1\right\rangle(\tau)\big|\\
&\lesssim \big(1+\|({\bm v}^{n-1}-\bar{\bm v})(\tau)\|_{\mathcal{H}^k(\Omega)}^9\big)\big\|({\bm v}^n-\bar{\bm v})(\tau)\big\|_{\mathcal{H}^k(\Omega)}
\big(1+\|({\bm v}^n-\bar{\bm v})(\tau)\|_{\mathcal{H}^k(\Omega)}\big).
\end{aligned}\end{align}

Next, to deal with the terms involving $\partial_\eta\partial^\alpha {\bm v}^{n-1}$ in the first part of \eqref{div-F1}, one has that by integration by parts,
\begin{align*}
&\left\langle\partial^\alpha({\bm v}^n-\bar{\bm v})^{\it T}\bm{S}({\bm v}^{n-1}),
\bm{F}({\bm v}^{n-1},\partial_\eta\partial^\alpha {\bm v}^{n-1})\partial_\eta {\bm v}^n\right\rangle(\tau)\nonumber\\
&=-\int_\bT\partial^\alpha({\bm v}^n-\bar{\bm v})^{\it T}[\bm{SF}]({\bm v}^{n-1}, \partial^\alpha {\bm v}^{n-1})\partial_\eta {\bm v}^n\big|_{\eta=0}d\xi\nonumber\\
&\quad-\left\langle\partial_\eta\partial^\alpha({\bm v}^n-\bar{\bm v})^{\it T}\bm{S}({\bm v}^{n-1}),
\bm{F}({\bm v}^{n-1}, \partial^\alpha {\bm v}^{n-1})\partial_\eta {\bm v}^n\right\rangle(\tau)\nonumber\\
&\quad-\left\langle\partial^\alpha({\bm v}^n-\bar{\bm v})^{\it T}\bm{S}({\bm v}^{n-1}), \bm{F}({\bm v}^{n-1},
\partial^\alpha {\bm v}^{n-1})\partial_\eta^2 {\bm v}^n\right\rangle(\tau)\nonumber\\
&\quad-\left\langle\partial^\alpha({\bm v}^n-\bar{\bm v})^{\it T},
{\hat{\bm F}}({\bm v}^{n-1}, \partial_\eta {\bm v}^{n-1}, \partial^\alpha {\bm v}^{n-1})\partial_\eta {\bm v}^n\right\rangle(\tau),
\end{align*}
where ${\hat{\bm F}}$ is a matrix involving ${\bm v}^{n-1}, \partial_\eta {\bm v}^{n-1}, \partial^\alpha {\bm v}^{n-1}$, and is quadratic with respect to $\partial_\eta {\bm v}^{n-1}$ and  $\partial^\alpha {\bm v}^{n-1}$. Then, it follows that for $k\geq4,$
\begin{align}\label{J1}
\begin{aligned}
&\Big|\left\langle\partial^\alpha({\bm v}^n-\bar{\bm v})^{\it T}\bm{S}({\bm v}^{n-1}), \bm{F}({\bm v}^{n-1},
\partial_\eta\partial^\alpha {\bm v}^{n-1})\partial_\eta {\bm v}^n\right\rangle(\tau)\\
	&\quad+\int_\bT\partial^\alpha({\bm v}^n-\bar{\bm v})^{\it T}[\bm{SF}]({\bm v}^{n-1}, \partial^\alpha {\bm v}^{n-1})\partial_\eta {\bm v}^n\big|_{\eta=0}d\xi\Big|\\
&\lesssim \big(1+\|({\bm v}^{n-1}-\bar{\bm v})(\tau)\|_{H^3(\Omega)}^4\big)
\big(1+\|\partial^\alpha ({\bm v}^{n-1}-\bar{\bm v})(\tau,\cdot)\|_{L^2(\Omega)}\big)\\
	&\quad\cdot
	\Big[	\|\partial_\eta\partial^\alpha({\bm v}^n-\bar{\bm v})(\tau,\cdot)\|_{L^2(\Omega)}\|\partial_\eta {\bm v}^n(\tau,\cdot)\|_{L^\infty(\Omega)}\\
	&\quad\quad+\|\partial^\alpha({\bm v}^n-\bar{\bm v})(\tau,\cdot)\|_{L^2(\Omega)}
	\big(\|\partial_\eta {\bm v}^n(\tau,\cdot)\|_{L^\infty(\Omega)}+\|\partial_{\eta}^2 {\bm v}^n(\tau,\cdot)\|_{L^\infty(\Omega)}\big)\Big]\\
&\lesssim \big(1+\|({\bm v}^{n-1}-\bar{\bm v})(\tau)\|_{\mathcal{H}^k(\Omega)}^5\big)\big(1+\|({\bm v}^n-\bar{\bm v})(\tau)\|_{\mathcal{H}^k(\Omega)}\big)\\
	&\qquad\cdot\big(\|({\bm v}^n-\bar{\bm v})(\tau)\|_{\mathcal{H}^k(\Omega)}
+\|\partial_\eta({\bm v}^n-\bar{\bm v})(\tau)\|_{\mathcal{H}^k(\Omega)}\big).
\end{aligned}\end{align}
Thus, combining \eqref{R-est} with \eqref{J1} and using \eqref{div-F1} we obtain Claim \ref{claim2}.
\end{proof}

To complete the construction of $n-$th order approximate solution $\bm v^n(\tau,\xi,\eta), n\geq1$, we need to check that such approximation $\bm v^n=(u_1^n, \theta^n,q^n)^{\it T}(\tau,\xi,\eta)$ solved from \eqref{picard} satisfies the corresponding bounded condition to $\eqref{induc-hypo}_3$, i.e.,
\[
\theta^n(\tau,\xi,\eta)\geq\delta,\quad \delta\leq q^n(\tau,\xi,\eta)\leq P(\tau,\xi)-\delta.
\]
In fact, we are able to obtain a stronger result as follows.
\begin{prop}\label{uniform-bound}
Under the assumptions of Theorem \ref{classical-sol}, 
there exist some $T_1\in(0,T_0]$ and some positive constant $M$ such that the approximate solution sequence
$$\{{\bm v}^n\}_{n\geq0}=\big\{(u_1^n,\theta^n,q^n)(\tau,\xi,\eta)\big\}_{n\geq0}$$
constructed above satisfies 
\begin{equation}\label{nth-bound}
\|{\bm v}^n-\bar{\bm v}\|_{\mathcal{H}^k(\Omega_{T_1})}\leq M,\quad \forall n\geq0,
\end{equation}
and
\begin{equation}\label{nth-lowbound}
\theta^n(\tau,\xi,\eta)\geq\delta,\quad \delta\leq q^n(\tau,\xi,\eta)\leq P(\tau,\xi)-\delta,\quad \forall (\tau,\xi,\eta)\in\Omega_{T_1}.
\end{equation}
\end{prop}

\begin{proof}
Firstly, it follows from \eqref{est_V0} that \eqref{nth-bound} holds for $n=0$ provided $M$ being sufficiently large.
Let us fix a $\displaystyle \widetilde{T}_1=\min\big\{T_0,\frac{1}{9C_0^{10}M_1^{27}}\big\}$ with the positive constant $C_0$ given in \eqref{gronwall-general}, and suppose that for $n\geq1,$
\begin{equation}\label{asmp-n-1}
\|{\bm v}^{n-1}-\bar{\bm v}\|^2_{\mathcal{H}^k(\Omega_t)}
\leq\frac{C_0M_1^3}{\big(1-9C_0^{10}M_1^{27}~t\big)^{\frac{1}{9}}},
\quad \forall\ t\in[0,\widetilde{T}_1).
\end{equation}
It is easy to know (\ref{asmp-n-1}) holds for the zero-th order approximate solution $\bm v^0$ provided $C_0$ being large.
Recalling \eqref{gronwall-general}, one can obtain that for $t\in[0,\widetilde T_1)$,
\begin{align}\label{induce-est}
\begin{aligned}
	\|\bm v^{n}-\bar{\bm v}\|^2_{\mathcal{H}^k(\Omega_t)}&\leq C_0M_1^3\exp\Big\{C_0\int_0^t\frac{C_0^9M_1^{27}}{\big(1-9C_0^{10}M_1^{27}~\tau\big)^{\frac{1}{9}}}d\tau\Big\}\\
	&\leq\frac{C_0M_1^3}{\big(1-9C_0^{10}M_1^{27}~t\big)^{\frac{1}{9}}}.
\end{aligned}\end{align}
On the other hand, under the assumption (\ref{outer-bound}) and (\ref{initial-lowbound}), it follows from (\ref{induce-est}) that
\begin{align*}
\theta^n(t,\xi,\eta)&=\theta^n_0(\xi,\eta)+\int^t_0\partial_\tau\theta^n(\tau,\xi,\eta)d\tau\geq2\delta-t~\|\partial_\tau \theta^n\|_{L^\infty(\Omega_t)}\\
&\geq2\delta-t\Big(\|{\bm v}^n-\bar{\bm v}\|_{\mathcal{H}^3(\Omega_t)}+\sup_{0\leq\tau\leq t}\|(\Theta,\theta^\ast)(\tau,\cdot)\|_{H^2(\bT_\xi)}\Big)\\
&\geq2\delta-2t\sqrt{C_0M_1^3}\big(1-9C_0^{10}M_1^{27}~t\big)^{-\frac{1}{18}},\quad \forall t\in[0,\widetilde{T}_1).
\end{align*}
Similar estimates hold for $q^n$ and $P-q^n$. Thus, choosing $T_1\in(0,\widetilde{T}_1)$
such that
\begin{align}\label{def-T1}
2T_{1}\sqrt{C_0M_1^3}\big(1-9C_0^{10}M_1^{27}~T_1\big)^{-\frac{1}{18}}\leq\delta,
\end{align}
then it implies that
\begin{equation*}
\theta^n(t,\xi,\eta)\geq\delta,\quad \delta\leq q^n(t,\xi,\eta)\leq P(\tau,\xi)-\delta.
\end{equation*}
Eventually, for such $T_1$ satisfying \eqref{def-T1}, let 
$$M=\sqrt{C_0M_1^3}\big(1-9C_0^{10}M_1^{27}~T_1\big)^{-\frac{1}{18}},$$  
 we can get the proof of the proposition.
\end{proof}

It follows from Proposition~\ref{uniform-bound} that, for the approximation solutions $\{{\bm v}^n\}_{n\geq0}$ to the problem (\ref{picard}) constructed above,
$\{{\bm v}^n-\bar{\bm v}\}_{n\geq 0}$ are uniform bounded in $\mathcal{H}^k(\Omega_{T_1}), k\geq4.$ 
Next, we will show the compactness of sequence $\{{\bm v}^n\}_{n\geq0}$. Indeed, it holds that
\begin{prop}\label{Vn-Vn-1}
There exists some $T_2\in(0,T_1]$ such that for $t\in[0,T_2],$
\begin{equation*}
\sup_{0\leq \tau \leq t}\big\|(\bm v^{n+1}-{\bm v}^n)(\tau,\cdot)\big\|_{L^2(\Omega)}
\leq\frac{1}{2}\sup_{0\leq \tau\leq t}\big\|({\bm v}^n-{\bm v}^{n-1})(\tau,\cdot)\big\|_{L^2(\Omega)}, \quad n\geq1.
\end{equation*}
\end{prop}

\begin{proof}
Denote
\[
\Psi^n:={\bm v}^{n+1}-{\bm v}^n=(\varphi^n, \psi^n,\omega^n)^{\it{T}},\quad n\geq0.
\]
For $n\geq1$, it follows from (\ref{picard}) that $\Psi^n=\Psi^n(\tau,\xi,\eta)$ satisfies the following initial-boundary value problem in $\Omega_{T_1}$,
\begin{equation}\label{Psi-simp}
\left\{
\begin{aligned}
&\partial_\tau\Psi^n+\bm{A}({\bm v}^n)\partial_\xi\Psi^n
+\bm{F}({\bm v}^n,\partial_\eta {\bm v}^n)\partial_\eta\Psi^n+\bm{G}({\bm v}^n)\Psi^n
-\bm{B}({\bm v}^n)\partial_{\eta}^2\Psi^n\nonumber\\
&\quad=\big[\bm{B}({\bm v}^n)-\bm{B}({\bm v}^{n-1})\big]\partial_{\eta}^2 {\bm v}^n-\big[\bm{A}({\bm v}^n)-\bm{A}({\bm v}^{n-1})\big]\partial_\xi {\bm v}^n\\
&\qquad-\big[\bm{F}({\bm v}^n,\partial_\eta {\bm v}^n)-\bm{F}({\bm v}^{n-1},\partial_\eta {\bm v}^{n-1})\big]\partial_\eta {\bm v}^n
-\big[\bm{G}({\bm v}^n)-\bm{G}({\bm v}^{n-1})\big] {\bm v}^n,\\
&(\varphi^n,\psi^n,\partial_\eta\omega^n)|_{\eta=0}=\mathbf{0},\quad
\lim_{\eta\rightarrow +\infty}\Psi^n(\tau,x,\eta)=\mathbf{0},\quad \Psi^n|_{\tau=0}=\mathbf{0}.
\end{aligned}
\right.
\end{equation}
Similar to the arguments as in the proof of Lemma \ref{energy-est}, we can claim that
\begin{align}\label{pertu}
\begin{aligned}
&\frac{1}{2}\frac{d}{d\tau}\left\langle(\Psi^n)^{\it T}\bm{S}({\bm v}^n), \Psi^n\right\rangle(\tau)
+\|\partial_\eta\Psi^n(\tau)\|^2_{L^2(\Omega)}\\
&\lesssim \mathcal{P}\Big(\|({\bm v}^n-\bar{\bm v})(\tau)\|_{\mathcal{H}^4(\Omega)},\|({\bm v}^{n-1}-\bar{\bm v})(\tau)\|_{\mathcal{H}^4(\Omega)}\Big)
\Big(\|\Psi^{n-1}(\tau)\|^2_{L^2(\Omega)}+\|\Psi^n(\tau)\|^2_{L^2(\Omega)}\Big),
\end{aligned}\end{align}
where $\mathcal{P}(\cdot, \cdot)$ denotes some polynomial. 
Then, from the uniform estimates in Proposition \ref{uniform-bound}
and the positive definiteness of $\bm S({\bm v}^n)$,
applying the Gronwall inequality to (\ref{pertu}) leads to that for $t\in[0,T_1]$,
\begin{align*}
\sup_{0\leq \tau\leq t}\|\Psi^n(\tau,\cdot)\|^2_{L^2(\Omega)}+\int_{0}^{t}\|\partial_\eta\Psi^n(\tau)\|^2_{L^2(\Omega)}d\tau
\lesssim t\sup_{0\leq \tau\leq t}\|\Psi^{n-1}(\tau,\cdot)\|^2_{L^2(\Omega)}.
\end{align*}
Therefore, there exists a proper $T_2\in(0,T_1]$ such that
\begin{align*}
\sup_{0\leq \tau\leq t}\|\Psi^n(\tau,\cdot)\|_{L^2(\Omega)}
\leq \frac{1}{2}\sup_{0\leq \tau\leq t}\|\Psi^{n-1}(\tau,\cdot)\|_{L^2(\Omega)},\quad\forall t\in[0,T_2],
\end{align*}
and we complete the proof of the proposition.
\end{proof}
\subsection{Local-in-time existence and uniqueness}
We now  prove Theorem~\ref{classical-sol} as follows.
\begin{proof}[Proof of Theorem~\ref{classical-sol}]
It follows from Proposition~\ref{uniform-bound} and Proposition~\ref{Vn-Vn-1} that the approximation sequence $\{{\bm v}^n-\bar{\bm v}\}_{n\geq0}$
is a Cauchy sequence in $L^\infty(0,T_\ast;L^2(\Omega))$ with $T_\ast=T_2$ given in Proposition~\ref{Vn-Vn-1}.
Hence, there exists $\bm{\tilde v}=(v_1,\vartheta,w)^{\it T}\in L^\infty(0,T_\ast;L^2(\Omega))$ such that
\[
\lim_{n\rightarrow +\infty}({\bm v}^n-\bar{\bm v})=\bm{\tilde v} \quad \mathrm{in}\quad L^\infty(0,T_\ast;L^2(\Omega)).
\]
From the boundedness given in Proposition~\ref{uniform-bound}, Fatou property for $\mathcal{H}^k(\Omega_{T_\ast})$ guarantees that
$\bm{\tilde v}\in \mathcal{H}^k(\Omega_{T_\ast})$.

On the other hand, by interpolation the sequence
$\{{\bm v}^n-\bar{\bm v}\}_{n\geq 0}$ also converges to $\bm{\tilde v}$ in $\mathcal{H}^{k'}(\Omega_{T_\ast})$ for any $k'< k$.
Thus, 
\begin{equation*}
{\bm v}=(u_1,\theta,q):=\bm{\tilde v}+\bar{\bm v}
\end{equation*}
is a classical solution to the problem (\ref{V-sys}) by letting $n\rightarrow \infty$ in (\ref{picard}).
Moreover, from \eqref{nth-lowbound} we have that for $\bm{v}$ such that
\begin{equation*}
\theta(\tau,\xi,\eta)\geq\delta,\ \delta\leq q(\tau,\xi,\eta)\leq P(\tau,\xi)-\delta,\quad \forall(\tau,\xi,\eta)\in\Omega_{T_\ast}.
\end{equation*}
Finally, the uniqueness of the solution to the problem (\ref{V-sys}) follows immediately from Proposition \ref{Vn-Vn-1}.

\end{proof}

\section{Existence for the original problem}\label{original}
In this section, we show that the solution $(\hat u_1,\hat\theta,\hat h_1)$ of the problem (\ref{lead-trans}) obtained in Corollary \ref{coro} can be transformed to a classical solution of the original initial-boundary value problem (\ref{simp-lead-noP}).

\begin{proof}[Proof of Theorem~\ref{original-sol}]
Firstly, we introduce $$\displaystyle \eta(x,y):=\int_{0}^{y}h_{1,0}(x,\tilde y)d\tilde y,$$ then $\eta\sim y$ since of the upper and lower bounds of $h_{1,0}$
given in \eqref{posi}. We denote by $y(x,\eta)$ the inverse function of $\eta(x,y)$ and define
	\begin{align}\label{def-ueta}
		(\hat u_{1,0},\hat\theta_0,\hat h_{1,0})(x,\eta):=(u_{1,0},\theta_0,h_{1,0})\big(x,y(x,\eta)\big).
	\end{align}
Let us consider the initial-boundary value problem \eqref{lead-trans}, where we replace the variable $(\tau,\xi)$ by $(t,x)$,
with the initial data $(\hat u_{1,0},\hat\theta_0,\hat h_{1,0})(x,\eta)$ in \eqref{def-ueta} and boundary condition $\theta^\ast(t,x)$.  Then, the assumptions on the initial boundary values in Theorem \ref{original-sol} and the fact $\eta\sim y$ imply the assumptions of Corollary \ref{coro}. Consequently, we know from
Corollary \ref{coro} that there exist a $T_\ast\in(0,T]$ such that \eqref{lead-trans} admits a unique classical solution
$(\hat u_1,\hat\theta,\hat h_1)(t,x,\eta)$ in $D_{T_\ast}$ satisfying
\begin{equation}\label{positi}
\hat\theta(t,x,\eta),~\hat h_{1}(t,x,\eta)\geq \delta,\quad
\frac{1}{2}\hat h_{1}^2(t,x,\eta) \leq P(t,x)-\delta, \quad \forall (t,x,\eta)\in D_{T_\ast}.
\end{equation}

We define $\psi=\psi(t,x,y)$ by the relation
\begin{equation}\label{invers}
y=\int^{\psi(t,x,y)}_0\frac{d\eta}{\hat h_1(t,x,\eta)}.
\end{equation}
Then $\psi$ is well-defined and continuous in $D_{T_\ast}$ by virtue of the continuity of $\hat{h}_1$ and $\hat{h}_1\geq\delta$ given in \eqref{positi}.
 Note that from \eqref{invers},
 \begin{align}\label{par-psi}
   \partial_y \psi(t,x,y)=\hat{h}_1\big(t,x,\psi(t,x,y)\big),
 \end{align}
 which along with the upper and lower bounds of $\hat{h}_1$ given in \eqref{positi}, implies that $\psi(t,x,y)\sim y$.
In other words, one has
\begin{align}\label{bd-psi}
\psi|_{y=0}=0,\quad \psi|_{y\rightarrow+\infty}\rightarrow+\infty.
\end{align}
Also, by using \eqref{def-ueta} it is easy to get
\begin{align}\label{ini-psi}
\psi(0,x,y)=\eta(x,y).
\end{align}
Moreover, direct calculation by using \eqref{invers} gives
\begin{equation}\label{psi-tx}
\left\{\begin{aligned}
&\partial_t \psi(t,x,y)=\hat h_1\big(t,x,\psi(t,x,y)\big)\int^{\psi(t,x,y)}_0\frac{\partial_t\hat  h_1}{\hat h_1^2}(t,x,\eta)d\eta,\\
&\partial_x \psi(t,x,y)=\hat h_1\big(t,x,\psi(t,x,y)\big)\int^{\psi(t,x,y)}_0\frac{\partial_x \hat  h_1}{\hat h_1^2}(t,x,\eta)d\eta,\\
&\partial_{y}^2\psi(t,x,y)=\hat h_1\big(t,x,\psi(t,x,y)\big)\partial_\eta\hat h_1\big(t,x,\psi(t,x,y)\big),\\
& \partial_y^3\psi(t,x,y)=\hat h_1\big(t,x,\psi(t,x,y)\big)
\partial_\eta\big((\hat h_1\partial_\eta\hat h_1)\big(t,x,\psi(t,x,y)\big)\big).
\end{aligned}
\right.
\end{equation}

Next, we define
\begin{equation}\label{(u1,theta,h1)}
(u_1,\theta,h_1)=(u_1,\theta,h_1)(t,x,y):=(\hat u_1,\hat\theta,\hat h_1)\big(t,x,\psi(t,x,y)\big).
\end{equation}
Particular, it holds from \eqref{par-psi} that 
\begin{align}\label{h1}
  h_1(t,x,y)~=~\partial_y\psi(t,x,y).
\end{align}
The boundedness in \eqref{positi} also holds for $(\theta, h_1)$. Then, set
\begin{equation}\label{(u2,h2)-trans}
u_2(t,x,y)=-\frac{(\partial_t\psi+u_1\partial_x\psi-\nu\partial_{y}^2\psi)}{h_1}(t,x,y),\quad h_2(t,x,y)=-\partial_x\psi(t,x,y).
\end{equation}
In the following, we will show that $(u_1, u_2, \theta, h_1, h_2)(t,x,y)$, defined in \eqref{(u1,theta,h1)} and \eqref{(u2,h2)-trans}, is the desired solution of Theorem \eqref{original-sol}.

It is straightforward to verify by using \eqref{psi-tx} that,  $(u_1,\theta,h_1), \partial_y(u_1,\theta,h_1)$
and $\partial_y^2(u_1,\theta,h_1)$ are continuous and bounded in $D_{T_\ast}$, 
$\partial_t(u_1,\theta,h_1), \partial_x(u_1,\theta,h_1)$ and $(u_2,h_2), \partial_y(u_2,h_2)$ are continuous and bounded on any compact sets of $D_{T_\ast}$.
Then, from \eqref{bd-psi} and the boundary conditions of $(\hat{u}_1, \hat{\theta}_1, \hat{h}_1)$ given in \eqref{lead-trans}, we can calculate that
\begin{equation*}
(u_1, u_2, \partial_y h_1, h_2)|_{y=0}=\mathbf{0},\quad \theta|_{y=0}=\theta^\ast(t,x),\quad
\lim_{y\rightarrow +\infty}(u_1,\theta,h_1)=(U,\Theta,H)(t,x).
\end{equation*}
Substituting \eqref{def-ueta} and \eqref{ini-psi} into \eqref{(u1,theta,h1)} gives
the initial condition
$$(u_1,\theta,h_1)|_{t=0}=(\hat u_{1,0},\hat\theta_0,\hat h_{1,0})(\xi,\eta),$$ 
and therefore, we know that $(u_1, u_2, \theta, h_1, h_2)(t,x,y)$ satisfies the initial boundary values of problem \eqref{simp-lead-noP}.

It remains to show $(u_1,u_2,\theta,h_1,h_2)$ satisfies the equations of the original system (\ref{simp-lead-noP}). Note that
the relation 
$$\partial_x h_1+\partial_y h_2=0$$ follows immediately from \eqref{h1} and the definition of $h_2$ given in \eqref{(u2,h2)-trans}.
Differentiating the first equation of $(\ref{(u2,h2)-trans})$ with respect to $y$, it follows from (\ref{par-psi}) and (\ref{psi-tx}) that
\begin{equation}\label{u1x+u2y}
h_1(\partial_x u_1+\partial_y u_2)
=-\partial_\tau\hat h_1+\hat h_1\partial_\xi\hat u_1-\hat u_1\partial_\xi\hat h_1
+\nu\hat h_1\big(\partial_\eta(\hat h_1\partial_\eta\hat  h_1)-(\partial_\eta\hat h_1)^2\big).
\end{equation}
By using the governing equation $(\ref{lead-trans})_3$ for $\hat h_1$, the above identity (\ref{u1x+u2y}) is reduced to
\begin{align*}
h_1(\partial_xu_1+\partial_yu_2)
&=\frac{(1-a)\hat h_1^2}{Q}\hat h_1\partial_\xi\hat  u_1
+\frac{\nu(1-a)\hat h_1^2}{Q}h_1\partial_\eta(\hat h_1\partial_\eta\hat h_1)-\frac{(1-a)\hat h_1}{Q}( P_\tau+u_1 P_\xi)\\
&\quad+\frac{a\hat h_1}{Q}
\big[\kappa\hat h_1\partial_\eta(\hat h_1\partial_\eta\hat \theta)+\nu(\hat h_1\partial_\eta\hat h_1)^2+\mu(\hat h_1\partial_\eta\hat  u_1)^2\big]\\
&=\frac{(1-a)h_1^2}{Q}( h_1\partial_x+h_2\partial_y)u_1
+\frac{\nu(1-a) h_1^2}{Q}\partial_{y}^2h_1-\frac{(1-a)h_1}{Q}(P_t+u_1P_x)\\
&\quad+\frac{ah_1}{Q}\big[\kappa\partial_{y}^2\theta+\nu(\partial_yh_1)^2+\mu(\partial_yu_1)^2\big],
\end{align*}
where we have used the facts:
  $$\hat h_1\partial_\eta=\partial_y,\ \hat h_1\partial_\xi=h_1\partial_x+h_2\partial_y.$$
Consequently, we get the equation $(\ref{simp-lead-noP})_4$.

Next, the governing equations of $(u_1,\theta,b_1)$ defined in (\ref{(u1,theta,h1)}) can be shown to satisfy the equations $(\ref{simp-lead-noP})_1$-$(\ref{simp-lead-noP})_3$ respectively.
For instance, it follows from (\ref{psi-tx}) and $(\ref{(u2,h2)-trans})_1$ that
\begin{align*}
\partial_tu_1+(u_1\partial_x+u_2\partial_y)u_1
=\partial_\tau\hat u_1+\hat u_1\partial_\xi\hat u_1+\nu\hat h_1\partial_\eta\hat h_1\partial_\eta\hat u_1.
\end{align*}
By the equation $(\ref{lead-trans})_1$, the above identity is reduced to
\begin{align*}
\partial_tu_1+(u_1\partial_x+u_2\partial_y)u_1
&=\frac{R\hat\theta}{P-\frac{1}{2}\hat h^2_1}\hat h_1\partial_\xi\hat h_1
+\frac{\mu R\hat\theta}{P-\frac{1}{2}\hat h^2_1}\hat h_1\partial_\eta(\hat h_1\partial_\eta\hat u_1)
-\frac{R\hat\theta P_\xi}{P-\frac{1}{2}\hat h^2_1}\\
&=\frac{R\theta}{P-\frac{1}{2}h^2_1}(h_1\partial_x+h_2\partial_y)h_1
+\frac{\mu R\theta}{P-\frac{1}{2}h^2_1}\partial_{y}^2u_1-\frac{R\theta P_x}{P-\frac{1}{2}h^2_1},
\end{align*}
which yields $(\ref{simp-lead-noP})_1$. Thus, we obtain that $(u_1, u_2, \theta, h_1, h_2)$,
defined by \eqref{(u1,theta,h1)} and \eqref{(u2,h2)-trans}, is a classical solution of problem $(\ref{simp-lead-noP})$.
Finally, the uniqueness of the solution to (\ref{simp-lead-noP}) follows from that of the transformed
problem (\ref{lead-trans}).
\end{proof}

\bigskip
\noindent {\bf Acknowledgement.}
 The research of the second author was sponsored by National Natural Science Foundation of China (Grant No. 11743009), Shanghai Sailing Program (Grant No. 18YF1411700) and Shanghai Jiao Tong University (Grant No. WF220441906).  The research of the third author was sponsored by General Research Fund of Hong Kong, CityU No. 11320016.

\bibliography{commhdref}
\end{document}